\documentclass[12pt,a4paper]{amsart}

\usepackage{amsmath,amssymb,amsfonts,amsthm}
\usepackage[top=35mm,bottom=35mm,left=30mm,right=30mm]{geometry}
\usepackage{enumitem}
\usepackage{mathptmx}
\usepackage{mathrsfs}
\usepackage[colorlinks=true,linkcolor=blue,citecolor=blue,urlcolor=blue]{hyperref}

\newtheorem{theorem}{Theorem}[section]
\newtheorem{proposition}[theorem]{Proposition}
\newtheorem{lemma}[theorem]{Lemma}
\newtheorem{corollary}[theorem]{Corollary}

\theoremstyle{definition}
\newtheorem{definition}[theorem]{Definition}
\newtheorem{example}[theorem]{Example}
\newtheorem*{step1}{Step 1}
\newtheorem*{step2}{Step 2}

\newtheorem*{conjecture}{Conjecture}

\theoremstyle{remark}
\newtheorem{remark}[theorem]{Remark}

\numberwithin{equation}{section}
\begin{document}
	\hbadness=2000

	\title{Variational Principles  for $r$-Neutralized Pfister-Sullivan Entropy}

	\author[E. Chen, Y. Xie and X. Zhou]{Ercai Chen, Yunxiang Xie and Xiaoyao Zhou}

	\address{School of Mathematical Sciences, Ministry of Education Key Laboratory of
		NSLSCS, Nanjing Normal University, Nanjing 210023, Jiangsu, P.R. China}
	\email{ecchen@njnu.edu.cn}
	\email{yxxie20@126.com}
	\email{zhouxiaoyaodeyouxian@126.com}



	\begin{abstract}
		In this paper, we study $r$-neutralized Pfister--Sullivan entropy, defined by measuring the complexity of orbit segments whose empirical measures lie near a prescribed invariant measure.  We mainly obtain the following three results:
		\begin{enumerate}
			\item[(i)] We disprove a variational conjecture   posed by C.~Dong and Q.~Qiao
			[\emph{On $r$-neutralized entropy: entropy formula and existence of measures
				attaining the supremum},
			Comm. Math. Phys. \textbf{406} (2025), Article No.~74.]

			\item[(ii)] We obtain variational formulas
			over probability measures for both upper and lower
			$r$-neutralized topological entropies.
			We also introduce new measure-theoretic quantities
			through PS localization and
			 establish some variational principles for upper
			$r$-neutralized topological entropy

			\item[(iii)] We obtain upper bounds for $r$-neutralized PS entropy
			under different smoothness assumptions and lower bounds
			for $C^{1+\alpha}$ diffeomorphisms, in terms of entropy,
			dimension and Lyapunov exponents.
			For $C^\infty$ diffeomorphisms and $C^{1}$ Anosov
			diffeomorphisms, the $r$-neutralized PS entropies equal
			$h_\mu(f)+r\dim M$ for every ergodic measure $\mu$.
			We also establish exact formulas on mixing surface
			hyperbolic basic sets and show that the whole-manifold
			measure formula can fail for $C^2$ surface diffeomorphisms.
		\end{enumerate}
	\end{abstract}

	\maketitle

	\section{Introduction and main results}\label{sec 1}

	The complexity of a dynamical system can be studied through the growth of its orbit segments and the geometry of the sets they form. Entropy measures how rapidly the number of distinguishable orbit segments increases with time. Dimension, on the other hand, describes the structure of invariant sets and measures as the observation scale becomes finer. Studying these quantities together can reveal how temporal complexity is reflected in small-scale geometry.

	For a continuous map $f:X\to X$ on a compact metric space, the classical variational principle \cite{Wal82} relates topological entropy to measure-theoretic entropy:
	\begin{align}\label{equ 1.1}
		h_{top}(f)
		=\sup_{\mu\in\mathcal M_f(X)}h_\mu(f),
	\end{align}
	where $\mathcal M_f(X)$ is the set of all $f$-invariant measures. Katok \cite{Kat80}  provided a  topological equivalent definition for measure-theoretic entropy.
	Brin and Katok \cite{BK83} gave a local characterization of measure-theoretic entropy via the exponential decay rate of the measures of Bowen balls.
In their study of saturated sets, Pfister and Sullivan \cite{PS07} investigated the orbit complexity of points whose empirical measures have prescribed accumulation points. Later,  Lindenstrauss and Tsukamoto \cite{LT18} connected  rate-distortion functions with measure-theoretic entropy from an information-theoretic perspective.

	The relation between entropy and dimension is particularly visible
	in smooth dynamics. Ruelle's inequality \cite{Rue78} bounds
	measure-theoretic entropy by the sum of positive Lyapunov exponents,
	while Pesin's entropy formula \cite{Pes77,LY85a} gives equality
	under suitable assumptions. On two-dimensional surfaces, Young \cite{You82}
	showed that, for a hyperbolic ergodic measure with Lyapunov
	exponents $\lambda_1>0>\lambda_2$,
	\begin{align*}
		\dim_H\mu
		=h_\mu(f)\left(\frac1{\lambda_1}-\frac1{\lambda_2}\right).
	\end{align*}
	In higher dimensions, Ledrappier and Young \cite{LY85a,LY85b}
	related entropy to dimensions along invariant directions. In 1985, Eckmann and Ruelle \cite{ER85} conjectured that every hyperbolic invariant measure of a $C^{1+\alpha}$ diffeomorphism is exact dimensional. That is, its pointwise dimension exists and is constant almost everywhere.
	Barreira, Pesin and Schmeling \cite{BPS99} subsequently proved
	the exact dimensionality of hyperbolic measures, resolving the
	Eckmann--Ruelle conjecture \cite{ER85}.

Estimating the dimension of a measure requires understanding how the measure of small neighborhoods scales along typical orbits. For the fixed radius, Bowen balls may be difficult to describe in the presence
	of nonuniform behavior. Ben Ovadia and Rodriguez-Hertz \cite{OH24}
	therefore considered Bowen balls with exponentially shrinking
	radius $e^{-nr}$ and introduced neutralized local entropy. They proved that
	neutralized local entropy agrees almost everywhere with
	Brin--Katok local entropy. They also used it to establish a
	lower bound for the pointwise dimension of ergodic invariant
	measures without assuming hyperbolicity, complementing the
	upper bound of Ledrappier and Young \cite{LY85b}.
	In the topological setting, Yang, Chen and Zhou \cite{YCZ24},
	motivated by \cite{FH12}, established variational principles for
	neutralized Bowen entropy on compact subsets. For
	$C^{1+\alpha}$ diffeomorphisms, Ben Ovadia \cite{Ova26}
	obtained an unstable entropy formula. Dong and Qiao \cite{DQ25} studied $r$-neutralized local
	entropy for invariant measures and obtained a formula between entropy and dimension
	for hyperbolic measures. They also proposed the following
	variational conjecture.

	\begin{conjecture}[\cite{DQ25}, Conjecture 1]
		Let $(X,d)$ be a compact metric space and let $f:X\to X$
		be a homeomorphism. Then, for every $r>0$,
		\begin{align*}
			\overline h_{top}^r(f,X)
			&=
			\sup_{\mu\in\mathcal M_f(X)}
			\int_X\limsup_{n\to\infty}
			-\frac{1}{n}\log\mu(B_n(x,e^{-nr}))\,\mathrm d\mu(x),\\
			\underline h_{top}^r(f,X)
			&=
			\sup_{\mu\in\mathcal M_f(X)}
			\int_X\liminf_{n\to\infty}
			-\frac{1}{n}\log\mu(B_n(x,e^{-nr}))\,\mathrm d\mu(x).
		\end{align*}
	\end{conjecture}
	We denote by $\overline h_\mu^{BK,r}(f)$ and
	$\underline h_\mu^{BK,r}(f)$  the two integrals on the
	right-hand side, respectively. Dong and Qiao \cite{DQ25}
	proved the conjecture for topological Markov shifts and linear
	Anosov diffeomorphisms on tori. Chen, K\"upper and Xie
	\cite{CKX25} showed that every expansive homeomorphism admits
	a compatible metric for which the conjecture holds when $r$
	is sufficiently small.

	Throughout the paper, $(X,f,d)$ denotes a topological dynamical
	system with $f$ a homeomorphism of the compact metric space
	$(X,d)$. Denote by $\mathcal M(X)$, $\mathcal M_f(X)$ and
	$\mathcal E_f(X)$ for the probability, invariant probability
	and ergodic invariant probability measures on $X$, respectively.
	Also, let $\mathcal D(X)$ be the set of metrics compatible with
	the topology of $X$.

	We first consider whether the conjecture might hold after choosing
	a suitable compatible metric and restricting to sufficiently
	small $r$. Our examples show that even this weaker requirement
	fails. Nevertheless, our first result shows that optimizing over probability measures at
	each finite time yields a variational principle.

	\begin{theorem}\label{thm 1.1}
		Let $(X,d)$ be a compact metric space and $f:X\to X$  be a homeomorphism. Alhough the	conjecture \cite{DQ25} fails in general, for each $r>0$, we can obtain the following variational principle:
		\begin{align*}
			\overline h_{top}^r(f,X)
			&=\limsup_{n\to\infty}
			\max_{\nu\in\mathcal M(X)}-\frac{1}{n} \int
			\log\nu(B_n(x,e^{-nr}))\,d\nu(x),\\
			\underline h_{top}^r(f,X)
			&=\liminf_{n\to\infty}
			\max_{\nu\in\mathcal M(X)}-\frac{1}{n} \int
			\log\nu(B_n(x,e^{-nr}))\,d\nu(x),
		\end{align*}
	\end{theorem}

	Theorem~\ref{thm 1.1} allows the maximizing probability measure
	to vary with $n$, whereas a variational principle over invariant
	measures requires a quantity attached to each fixed $\mu$.
	Following the empirical-measure approach of
	Pfister and Sullivan \cite{PS07}, we instead count orbit segments
	whose empirical measures lie in a neighborhood of $\mu$.
	Applying the same
	localization to neutralized Bowen ball  and rate
	distortion with fixed time and scale yields the following result.

	\begin{theorem}\label{thm 1.2}
		Let $(X,d)$ be a compact metric space and $f:X\to X$  be a homeomorphism. For each $r>0$ and
		$\mu\in\mathcal M_f(X)$, the $r$-neutralized localized Brin-Katok entropy and  localized rate-distortion quantity are equal with the $r$-neutralized PS entropy:
		\begin{align*}
			\mathcal H_\mu^r(f,X)=\mathcal R_\mu^r(f,X)
			=\overline h_\mu^{PS,r}(f,X).
		\end{align*}
		Moreover, we have
		\begin{align*}
			\overline h_{top}^r(f,X)
			=\max_{\mu\in\mathcal M_f(X)}
			\mathcal R_\mu^r(f,X)
			=\max_{\mu\in\mathcal M_f(X)}
			\overline h_\mu^{PS,r}(f,X).
		\end{align*}
	\end{theorem}
	The corresponding variational principle for the lower
	$r$-neutralized entropy fails, even though the lower
	probability-measure formulas remain valid. The example in
	Section~\ref{sec 2} also shows that the lower and upper PS
	entropies can differ for an ergodic measure.

We next estimate PS entropy for smooth dynamical systems.
For a hyperbolic ergodic measure $\mu$ of a
$C^{1+\alpha}$ diffeomorphism, Dong and Qiao \cite{DQ25}
proved
\begin{align}\label{equ 1.2}
	\underline h_\mu^{BK,r}(f)
	=\overline h_\mu^{BK,r}(f)
	=h_\mu(f)+r\dim_H\mu.
\end{align}
PS entropy also counts orbit segments that are not typical
with respect to $\mu$, so this formula does not determine
its dimension term.

	We first consider such segments in the whole manifold and relate their complexity to measure entropy, Lyapunov exponents and the dimension of the manifold.
	Let $f:M\to M$ be a $C^1$ diffeomorphism of a closed Riemannian $m$-manifold. For $\mu\in\mathcal M_f(M)$, let $\lambda_1(x)>\ldots>\lambda_{\ell(x)}(x)$ be the distinct Lyapunov exponents at a regular point $x$ with $d_i(x)$ the   multiplicity of $\lambda_i(x)$. We define the following four quantities:
	\begin{align*}
		\mathcal X^+(f,\mu)
		&= \int_M \sum_{j=1}^{\ell(x)} d_j(x)\max\{\lambda_j(x),0\}\,\mathrm d\mu(x),
		& \mathcal X_{\max}^+(f,\mu)
		&= \int_M \max\{\lambda_1(x),0\}\,\mathrm d\mu(x),\\[2mm]
		\mathcal X^-(f,\mu)
		&= \int_M \sum_{j=1}^{\ell(x)} d_j(x)\max\{-\lambda_j(x),0\}\,\mathrm d\mu(x),
		& \mathcal X_{\max}^-(f,\mu)
		&= \int_M \max\{-\lambda_{\ell(x)}(x),0\}\,\mathrm d\mu(x).
	\end{align*}
	Denote by \begin{align*}
		\Phi^*_1(f,\mu)=\min\left\{\mathcal X^+(f,\mu),\mathcal X^-(f,\mu)\right\},\Phi^*_2(f,\mu)=\min\left\{\mathcal X^+_{\max}(f,\mu),\mathcal X^-_{\max}(f,\mu)\right\}
	\end{align*}

	We obtain the following upper bounds, with a stronger estimate under
	higher smoothness.

	\begin{theorem}[Upper bound]\label{thm 1.3}
		Let $f:M\to M$ be a diffeomorphism of a closed Riemannian
		$m$-manifold, and let $r>0$ and $\mu\in\mathcal M_f(M)$.
		\begin{enumerate}
			\item If $f$ is $C^1$, then
			\begin{align*}
				\overline h_\mu^{PS,r}(f,M)
				\leq mr+\Phi_1^*(f,\mu).
			\end{align*}

			\item If $f$ is $C^k$ for some integer $k\geq2$, then
			\begin{align*}
				\overline h_\mu^{PS,r}(f,M)
				\leq mr+\min\biggl\{
				\Phi_1^*(f,\mu),
				h_\mu(f)+\frac{m}{k}\Phi_2^*(f,\mu)
				\biggr\}.
			\end{align*}
		\end{enumerate}
	\end{theorem}
	The two terms in the minimum of Theorem~\ref{thm 1.3}
	are not comparable in general, as illustrated by
	Examples \ref{exam 3.6}.
	For $\mu\in\mathcal E_f(M)$, set
	\begin{align}\label{equ 1.3}
		h_\mu^*(f)
		=\inf_{F\ni\mu}\
		\sup_{\nu\in F\cap\mathcal E_f(M)}h_\nu(f),
	\end{align}
	where $F$ ranges over neighborhoods of $\mu$ in
	$\mathcal M_f(M)$.

	\begin{theorem}(Lower bound)\label{thm 1.4}
		Let $f:M\to M$ be a $C^{1+\alpha}$ diffeomorphism of a closed
		Riemannian $m$-manifold and $r>0$. For every
		$\mu\in\mathcal E_f(M)$, we have
		\begin{align*}
			\underline h_\mu^{PS,r}(f,M)\geq  h_\mu^*(f)+mr\geq	h_\mu(f)+mr.
		\end{align*}
	\end{theorem}

	We cannot directly apply  \cite{DQ25}  to give  estimations for $r$-neutralized topological entropies, while
	the above estimates
	for the $r$-neutralized PS entropy, together with the variational principle,
	can give bounds for the $r$-neutralized topological entropy.
	If $f$ is $C^{1+\alpha}$, then  for every $r>0$,
	\begin{align}\label{equ 1.4}
		h_{top}(f)+r\dim M
		\leq\underline h_{top}^r(f,M)
		\leq\overline h_{top}^r(f,M)
		\leq r\dim M+
		\sup_{\mu\in\mathcal M_f(M)}\Phi_1^*(f,\mu).
	\end{align}
	For $C^{1+\alpha}$ or $C^k$ ($k\geq2$)  diffeomorphisms, we can also strengthen the upper bound  by replacing $\Phi_1^*(f,\mu)$ with the measure-theoretic quantity in Theorem \ref{thm 1.3}.

	The preceding bounds raise the question of when the
	$r$-neutralized PS entropy admits an exact formula in terms of
	measure entropy and the dimension of the manifold. We establish
	such a formula for two classes of systems: $C^\infty$
	diffeomorphisms and $C^1$ Anosov diffeomorphisms. In each case,
	the lower and upper PS entropies agree, and the corresponding
	topological entropies satisfy an analogous formula.

	\begin{theorem}\label{thm 1.5}
		Let $f:M\to M$ be either a $C^\infty$ diffeomorphism or a $C^1$ Anosov diffeomorphism of a closed Riemannian manifold. Then for every $r>0$ and $\mu\in\mathcal E_f(M)$, one has
		\begin{align*}
			\underline h_\mu^{PS,r}(f,M)
			=\overline h_\mu^{PS,r}(f,M)
			=h_\mu(f)+r\dim M,\\
				\underline h_{top}^r(f,M)
				=\overline h_{top}^r(f,M)
				=h_{top}(f)+r\dim M.
		\end{align*}
	\end{theorem}

In Example~\ref{exam 3.12}, we will show that the first forluma fails for $C^2$ diffeomorphism .

Compact invariant sets are central in smooth dynamics, arising as attractors, repellers, homoclinic closures, and general hyperbolic sets with rich recurrent and chaotic behavior. Their fractal geometry restricts available orbit segments and reveals fine structures invisible to the ambient manifold dimension alone. Understanding their dimensions is thus longstanding. McCluskey and Manning \cite{MM83} studied surface horseshoe dimensions along unstable manifolds; Takens \cite{Tak88} investigated Hausdorff dimension and limit capacity of dynamically defined Cantor sets; Palis and Viana \cite{PV88} examined how horseshoe dimensions vary with the diffeomorphism; Wang, Wang, Cao and Zhao \cite{WWCZ} derived dimension formulas for average conformal hyperbolic sets. The dimension of an invariant set and that of a supported measure need not agree. PS entropy counts orbit segments in $\Lambda$ whose empirical measures are close to $\mu$. Because these segments remain in $\Lambda$, the whole-manifold lower bound involving $\dim M$ does not apply directly. If $f$ is $C^{1+\alpha}$, for every $\mu\in\mathcal E_f(\Lambda)$,
	we nevertheless have
	\begin{align*}
		\underline h_\mu^{BK,r}(f)\leq 	\underline h_\mu^{PS,r}(f,\Lambda)
		\leq\overline h_\mu^{PS,r}(f,\Lambda)
		\leq r\dim M+
		\min\{\mathcal X^+(f,\mu),\mathcal X^-(f,\mu)\}.
	\end{align*}
	For hyperbolic $\mu$, \eqref{equ 1.2} gives the lower bound
	$h_\mu(f)+r\dim_H\mu$. These estimates also yield bounds for
	the topological entropies.
	A natural question is whether
	$\dim_H\Lambda$ can have an exact formula.   If
	$\Lambda$ is a mixing hyperbolic basic set of a $C^{1+\alpha}$
	surface diffeomorphism, then every $\mu\in\mathcal E_f(\Lambda)$
	satisfies
	\begin{align*}
		\underline h_\mu^{PS,r}(f,\Lambda)
		=\overline h_\mu^{PS,r}(f,\Lambda)
		=h_\mu(f)+r\dim_H\Lambda.
	\end{align*}
	The lower and upper topological entropies both equal
	$h_{top}(f|_\Lambda)+r\dim_H\Lambda$. Thus the PS formula
	involves the dimension of the set, while \eqref{equ 1.2}
	involves the dimension of the measure.

	The remainder of this paper is organized as follows. In Section~\ref{sec 2}, we disprove the variational conjecture   and establish the variational principles for $r$-neutralized topological entropies. In Section~\ref{sec 3}, we study the dimension estimates for the $r$-neutralized PS entropies on the whole  manifold   and obtain exact formulas for $C^{\infty}$ and Anosov diffeomorphisms. In Section~\ref{sec 4}, we study compact invariant sets and compare PS entropy with neutralized local entropy.

	\section{Counterexamples and variational principles for $r$-neutralized entropy}
	\label{sec 2}
	We first introduce some definitions of $r$-neutralized entropies and  give counterexamples to \cite[Conjecture~1]{DQ25}.  Then we establish some variational principles in terms of probability measures and invariant measures.

	\subsection{Counterexamples to the variational conjecture}
	\label{subsec 2.1}
	Let $(X,d)$ be a compact metric space and $f:X\to X$ be a homeomorphism. Fix $r>0$.
	For any $n\in \mathbb N$ and $x,y\in X$,  the $n$-th Bowen metric is  defined  as
	\begin{align*}
		d_n(x,y)=\max_{0\leq j<n}d(f^jx,f^jy).
	\end{align*}

	\begin{definition}
		Let $Z\subseteq X$ be a nonempty set.
		A set $E\subseteq X$ is called an $(n,e^{-nr})$-spanning
		set of $Z$ if $Z\subseteq \bigcup_{x\in E}B_n(x,e^{-nr})$. A set $F\subseteq Z$ is called an
		$(n,e^{-nr})$-separated set of $Z$ if for any distinct points $x,y\in F$, one has $d_n(x,y)\geq e^{-nr}$.
		Denote by $r_n(f,Z,e^{-nr})$ and $s_n(f,Z,e^{-nr})$
		the smallest
		cardinality of $(n,e^{-nr})$-spanning sets  and
		the largest cardinality of $(n,e^{-nr})$-separated sets of $Z$.	We use the convention that $s_n(f,\emptyset,e^{-nr})=r_n(f,\emptyset,e^{-nr})=1$.
	\end{definition}
	The \emph{upper and lower $r$-neutralized topological entropies of $Z\subseteq X$} are defined by \cite{DQ25}
	\begin{equation}\label{equ 2.1}
		\begin{split}
			\overline h_{top}^r(f,Z)=\limsup_{n\to\infty}\frac{1}{n}\log r_n(f,Z,d,e^{-nr})=\limsup_{n\to\infty}\frac{1}{n}\log s_n(f,Z,d,e^{-nr}),\\
			\underline h_{top}^r(f,Z)=\liminf_{n\to\infty}\frac{1}{n}\log r_n(f,Z,d,e^{-nr})=\liminf_{n\to\infty}\frac{1}{n}\log s_n(f,Z,d,e^{-nr}).
		\end{split}
	\end{equation}
	When $Z=X$, we omit $Z$ for simplicity.
	Although these quantities depend on $d$ compatible with the topology of $X$, we omit $d$ from the notation when no confusion can arise.

	\begin{remark}\label{rem 2.2}
		The two definitions in \eqref{equ 2.1} are equivalent.
		Choose an integer $q\geq1$ with $e^{-qr}<\frac{1}{2}.$
		Then by the standard comparison in \cite{YCZ24}, one has
		\begin{align*}
			r_n(f,Z,d,e^{-nr})
			\leq
			s_n(f,Z,d,e^{-nr})
			\leq
			r_n\left(f,Z,d,\frac{1}{2}e^{-nr}\right)
			\leq
			r_{n+q}(f,Z,d,e^{-(n+q)r}).
		\end{align*}
		Since $q$ is independent of $n$, the time shift $n\mapsto n+q$ does not affect either exponential growth rate.
	\end{remark}

	Recall that the box dimensions are defined as
	\begin{align*}
		\overline{\dim}_B(X,d)=\limsup_{\epsilon\to 0}\frac{\log N(X,\epsilon)}{-\log \epsilon},\quad \underline{\dim}_B(X,d)=\liminf_{\epsilon\to 0}\frac{\log N(X,\epsilon)}{-\log \epsilon},
	\end{align*}
	where $N(X,\epsilon)$ denotes smallest number of balls covering $X$. If they are equal, we denote by $\dim_B(X,d)$ their common value.
	Then we can observe the following basic facts, which are key tools in the examples:
	\begin{align}\label{equ 2.2}
		\overline h_{top}^r(f,Z)\geq r \overline{\dim}_B(Z,d),\quad \underline h_{top}^r(f,Z)\geq r \underline{\dim}_B(Z,d)
	\end{align}
	The relation between neutralized entropy (by considering $r\to 0$) and classical topological entropy has attracted considerable attention. The following  Example shows that they can differ, which also answers the open question of \cite{YCZ24}.

	\begin{example}\label{exam 2.3}
		Let $f:\mathbb S^1\to\mathbb S^1$ be the irrational rotation
		$f(x)=x+\alpha\pmod 1$, where $\alpha\notin\mathbb Q$.
		Equip $\mathbb S^1$ with its geodesic metric. For
		$\nu,\eta\in\mathcal M(\mathbb S^1)$, define
		\begin{align*}
			W_1(\nu,\eta)
			=\inf_{\pi\in\Pi(\nu,\eta)}
			\int_{\mathbb S^1\times\mathbb S^1}
			d(x,y)\,\mathrm d\pi(x,y),
		\end{align*}
		where $\Pi(\nu,\eta)$ is the set of probability measures on
		$\mathbb S^1\times\mathbb S^1$ with marginals $\nu$ and $\eta$. Consider the topological  dynamical system
		$(M(\mathbb S^1),f_*,W_1)$, where  $f_*:\mathcal M(\mathbb S^1)\to\mathcal M(\mathbb S^1)$
		is the induced map of $f$. Then one has
		\begin{align*}
			\lim_{r\to 0}\underline h_{top}^r
			(f_*,\mathcal M(\mathbb S^1))
			=\lim_{r\to 0}\overline h_{top}^r
			(f_*,\mathcal M(\mathbb S^1))
			=+\infty
			>h_{top}(f_*)=0.
		\end{align*}
	\end{example}

	\begin{proof}
		By \cite{GW95}, since $h_{top}(f)=0$, one has
		$h_{top}(f_*)=0$.
		Take distinct points $x_1,\ldots,x_N\in\mathbb S^1$ with $N\geq 2$.
		The set
		\begin{align*}
			\Delta_N
			=\left\{\sum_{i=1}^N p_i\delta_{x_i}:
			p_i\geq0,\ \sum_{i=1}^N p_i=1\right\}.
		\end{align*}
		is bi-Lipschitz equivalent   to the standard
		$(N-1)$-dimensional simplex under $W_1$ . Hence
		$\underline{\dim}_B\mathcal M(\mathbb S^1)\geq N-1$.
		Since $N$ is arbitrary,
		$\underline{\dim}_B\mathcal M(\mathbb S^1)=+\infty$.
		Then by \eqref{equ 2.2},
		\begin{align*}
			\underline h_{top}^r
			(f_*,\mathcal M(\mathbb S^1))
			=r\,\underline{\dim}_B
			(\mathcal M(\mathbb S^1),W_1)
			=+\infty.
		\end{align*}
	\end{proof}

	We then give an example to show that the conjecture of \cite{DQ25} fails in general.

	\begin{example}\label{exam 2.4}
		Let $	\mathbb S^m
		=
		\{x\in\mathbb R^{m+1}:\|x\|=1\}$
		with $m\geq 1$ and $d$ be
		the induced geodesic distance. Let $H(x)=x_{m+1}$ be the height function on $\mathbb S^m$. Since the smooth vector field $-\nabla H$ is complete, it generates a smooth flow $\{\varphi_t\}_{t\in\mathbb R}$ \cite[Chapter~9]{Lee13}.		Take $f:=\varphi_1$.
		Then $f$ is a $C^\infty$ north--south diffeomorphism of $\mathbb S^m$. Moreover, for   every $r>0$,
		\begin{align*}
			\overline h_{top}^r(f)
			\geq
			\underline h_{top}^r(f)
			\geq mr>0=\underline h_\mu^{BK,r}(f)
			=
			\overline h_\mu^{BK,r}(f),
		\end{align*}
		for every $\mu\in \mathcal M_f(\mathbb S^m)$.
	\end{example}

	\begin{proof}
		By \cite{Fal03}, one knows $\overline{\dim}_B(\mathbb S^m,d)=\underline{\dim}_B(\mathbb S^m,d)=m$. Then by \eqref{equ 2.2},
		\begin{align*}
			\overline h_{top}^r(f)
			\geq	\underline h_{top}^r(f)
			\geq
			mr>0.
		\end{align*}

		One can notice that  the recurrent points are only
		$p_-=(0,\ldots,0,1)$  $p_+=(0,\ldots,0,-1)$ \cite{GL08}.
		Then by the Poincar\'e recurrence theorem \cite{Wal82}, each $f$-invariant measure is of the form:  $\mu=c\delta_{p_-}+(1-c)\delta_{p_+}$, where $c\in [0,1]$.
		Then for $\mu$-a.e. $x$,
		$\mu\bigl(B_n(x,e^{-nr})\bigr)\geq\mu(\{x\})>0$.
		It follows that
		$	\underline h_\mu^{BK,r}(f)
		=
		\overline h_\mu^{BK,r}(f)
		=0.$
	\end{proof}

	The conjecture even
	fails for each compatible metric when $r$ is small enough.

	\begin{example}\label{exam 2.5}
		Let $Q=[0,1]^{\mathbb N}$ be the Hilbert cube, and let
		\begin{align*}
			X=(\mathbb Z\times Q)\cup\{\infty\}
		\end{align*}
		be the one-point compactification of $\mathbb Z\times Q$. Define the
		homeomorphism $f:X\to X$ by
		\begin{align*}
			f(k,q)=(k+1,q),
			\qquad
			f(\infty)=\infty.
		\end{align*}
		Then $(X,f)$ is uniquely ergodic and $\mathcal M_f(X)=\{\delta_\infty\}.$
		Moreover, for every $d\in \mathcal D(X)$ and  $r>0$, one has
		\begin{align*}
			\underline h_{top}^r(f)
			=
			\overline h_{top}^r(f)
			=+\infty,\quad \text{and}\quad
			\underline h_{\delta_\infty}^{BK,r}(f)
			=
			\overline h_{\delta_\infty}^{BK,r}(f)
			=0.
		\end{align*}

	\end{example}

	\begin{proof}
		It is easy  to check that 	$\infty$ is the unique recurrent point.	By Poincar\'{e} recurrence theorem again, we obtain  $\mathcal M_f(X)=\{\delta_\infty\}$. Since $\delta_\infty(\{\infty\})=1$, for every $n\geq1$, we have $\delta_\infty\bigl(B_n(\infty,e^{-nr})\bigr)=1.$
		Thus  $\underline h_{\delta_\infty}^{BK,r}(f)
		=
		\overline h_{\delta_\infty}^{BK,r}(f)
		=0.$

		Notice that
		$\{0\}\times Q$ is a closed subspace of $X$ and the Hilbert cube has infinite topological dimension. Then
		$	\underline{\dim}_B(X,d)\geq
		\dim X=+\infty.$
		By \eqref{equ 2.2},	one has
		$\underline h_{top}^r(f,d)
		\geq
		r\underline{\dim}_B(X,d)
		=+\infty.$
	\end{proof}

	We first give the proof of Theorem \ref{thm 1.1}

	\begin{proof}[Proof of Theorem~\ref{thm 1.1}]
		We first prove the lower bound.
		Fix  $\nu\in \mathcal M(X)$. Since
		$$\nu\longmapsto
		-\frac{1}{n}\int_X\log\nu(B_n(x,\epsilon))\,
		\mathrm d\nu(x)$$
		is upper semicontinuous on $\mathcal M(X)$, so the supremum
		can be attained for each $n$.

		Denote $B_i=B_n(x_i,\frac{e^{-nr}}{3})$.
		Consider the  the family of neutralized Bowen balls $\{B_i\}_{i=1}^N$ covering $X$ with the cardinality $r_n(f,X,\frac{e^{-nr}}{3})$. Set $P_1:=B_1$ and $P_i=B_i\setminus \bigcup_{k<i}B_k$ and suppose each set is nonempty. Then $\mathcal P=\left\{P_1,\ldots,P_N\right\}$
		form a Borel partition  of $X$
		Denote by $\mathcal P(x)$ the atom containing $x$. Then for each $x\in X$, $\mathcal P(x)\subseteq B_n(x,e^{-nr})$.
		Consequently,
		\begin{align*}
			-\frac{1}{n}\int_X\log\nu\bigl(B_n(x,e^{-nr})\bigr)\,\mathrm d\nu(x)
			\leq\frac{1}{n} H_\nu(\mathcal P)
			\leq\frac{1}{n}\log r_n\left(f,X,\frac{e^{-nr}}{3}\right).
		\end{align*}
		By Remark \ref{rem 2.2}, one knows that the parameter has no effect on the value of $r$-neutralized topological entropy. Hence,
		\begin{align*}
			\limsup_{n\to\infty}\max_{\nu\in\mathcal M(X)}
			\left[-\frac{1}{n}\int_X
			\log\nu\bigl(B_n(x,e^{-nr})\bigr)\,\mathrm d\nu(x)\right]
			\leq\overline h_{top}^r(f,X).
		\end{align*}

		For the reverse inequality,
		let $E_n$ be an
		$(n,e^{-nr})$-separated subset of $X$ with the largest
		cardinality, and take
		\begin{align*}
			\nu_n=\frac{\sum_{x\in E_n}\delta_x}{\#E_n}.
		\end{align*}
		Consequently,
		\begin{align*}
			\max_{\nu\in\mathcal M(X)}
			-\frac{1}{n}\int_X
			\log\nu\bigl(B_n(x,e^{-nr})\bigr)\,
			\mathrm d\nu(x)
			&\geq
			-\frac{1}{n}\int_X
			\log\nu_n\bigl(B_n(x,e^{-nr})\bigr)\,
			\mathrm d\nu_n(x)\\
			&=\frac{1}{n}\log s_n(f,X,e^{-nr}).
		\end{align*}
		Taking the upper limit  yields
		\begin{align*}
			\overline h_{top}^r(f,X)
			\leq
			\limsup_{n\to\infty}
			\max_{\nu\in\mathcal M(X)}
			-\frac{1}{n}\int_X
			\log\nu\bigl(B_n(x,e^{-nr})\bigr)\,
			\mathrm d\nu(x).
		\end{align*}

		The variational principle for  $\underline h_{top}^r(f,X)$ is similar, we omit the proof here.
	\end{proof}

	\begin{remark}
		Different with $\overline{h}_\mu^{BK,r}(f)$ and $\underline{h}_\mu^{BK,r}(f)$, the integral cannot in general be moved outside the  limit of $n$
		in Theorem \ref{thm 1.1}. As the maximum of probability measure is chosen separately for each $n$, moving the integral outside the upper limit would require a single measure to capture all scales.
	\end{remark}

	\subsection{$r$-neutralized PS entropy and Localized   entropies }
	Motivated by Pfister and Sullivan \cite{PS07}, we consider the orbit segments whose empirical measures lie in a neighborhood of an invariant measure.

	Let $F\subset \mathcal M(X)$. For $x\in X$, we denote
	\begin{align*}
		\mathcal E_n(x)=\frac{1}{n}\sum_{j=0}^{n-1}\delta_{f^jx},
		\quad X_{n,F}=\{x\in X:\mathcal E_n(x)\in F\}.
	\end{align*}

	\begin{definition}
		\label{def 2.7}
		Fix $\mu\in \mathcal M_f(X)$ and $r>0$. We define the \emph{upper and lower $r$-neutralized Pfister-Sullivan (PS) entropies} by
		\begin{align*}
			\overline h_{\mu}^{PS,r}(f,X)
			&=\inf_{F\ni\mu}\limsup_{n\to\infty}\frac{1}{n}
			\log s_n(X_{n,F},e^{-nr}),\\
			\underline h_{\mu}^{PS,r}(f,X)
			&=\inf_{F\ni\mu}\liminf_{n\to\infty}\frac{1}{n}
			\log s_n(X_{n,F},e^{-nr}),
		\end{align*}
		where the infimum  taks over all open neighborhoods $F$ of $\mu$ in $\mathcal M(X)$.
	\end{definition}

	By monotonicity with respect to $F$, the value is unchanged if the
	infimum is taken either over all  open neighborhoods  or  closed neighborhoods $F$ satisfying
	$\mu\in\operatorname{int}F$.

	\begin{remark}
		The same quantities are obtained by using spanning sets.
		Fix a neighborhood $F$ of $\mu$ and choose an integer
		$q\geq1$ with $2e^{-qr}<1$. Since
		$\mathcal E_{n+q}(x)-\mathcal E_n(x)\to0$ uniformly in $x$
		in the weak$^*$ topology, there exists a smaller neighborhood
		$F_0$ of $\mu$ such that
		$X_{n,F_0}\subseteq X_{n+q,F}$ for all sufficiently large $n$.
		Consequently,
		\begin{align*}
			r_n(f,X_{n,F_0},e^{-nr})\leq 	s_n(f,X_{n,F_0},e^{-nr})
			\leq r_{n+q}(f,X_{n+q,F},e^{-(n+q)r}).
		\end{align*}
	\end{remark}

	Different with the entropy map $\mu\longmapsto h_\mu(f)$,	the PS entropy is upper semicontinuous.
	\begin{lemma}\label{lem 2.9}
		The map $\mathcal M_f(X)\ni\mu
		\longmapsto
		\overline h_\mu^{PS,r}(f,X)$
		is upper semicontinuous.
		Then
		\begin{align*}
			\sup_{\mu\in \mathcal M_f(X)}\overline h_\mu^{PS,r}(f,X)=\max_{\mu\in \mathcal M_f(X)}\overline h_\mu^{PS,r}(f,X)
		\end{align*}

	\end{lemma}

	\begin{proof}
		Let $\mu\in \mathcal M_f(X)$ and
		$\overline h_\mu^{PS,r}(f,X)<a\in \mathbb R$. Then there exists an
		open neighborhood $F$ of $\mu$ in $\mathcal M(X)$ such that
		\begin{align*}
			\limsup_{k\to\infty}
			\frac1k\log r_k(f,X_{k,F},e^{-kr})<a.
		\end{align*}
		Choose an open neighborhood $V$ of $\mu$ such that $V\subset F$.
		For every $\lambda\in V\cap \mathcal M_f(X)$, $F$ is also an open
		neighborhood of $\lambda$. Hence
		\begin{align*}
			\overline h_\lambda^{PS,r}(f,X)
			\leq
			\limsup_{k\to\infty}
			\frac1k\log r_k(f,X_{k,F},e^{-kr})
			<a.
		\end{align*}
		Thus
		$\mu\mapsto\overline h_\mu^{PS,r}(f,X)$ is upper
		semicontinuous on $\mathcal M_f(X)$. The second result follows from that $M_f(X)$ is compact.
	\end{proof}

	In ergodic theory, there are several ways to characterize measure-theoretic entropy. In the following, we introduce a new  quantity from the  information theory  to characterize neutralized entropy. We first recall the notions of mutual information  in
	\cite{CT06,LT18}.

	Given a probability space $(\Omega,\mathbb{P})$. Let $X$ and $Y$ be measurable spaces, and $\xi:\Omega\to X$ and $\eta:\Omega\to Y$ measurable maps. The \emph{mutual information} $I(\xi,\eta)$ is defined as the supremum of
	\begin{equation*}
		\sum_{1\le m\le M,\,1\le n\le N}
		\mathbb{P}(\xi\in P_m,\eta\in Q_n)
		\log
		\frac{\mathbb{P}(\xi\in P_m,\eta\in Q_n)}
		{\mathbb{P}(\xi\in P_m)\mathbb{P}(\eta\in Q_n)},
	\end{equation*}
	where $\{P_1,\dots,P_M\}$ and $\{Q_1,\dots,Q_N\}$ are measurable partitions of $X$ and $Y$ respectively, with the convention that $0\log(0/a)=0$ for all $a\ge0$. If $\xi$ and $\eta$ have finite ranges, then  $I(\xi,\eta)$ is given by
	\begin{equation*}
		I(\xi,\eta)
		=
		H(\xi)-H(\xi\mid\eta)
		=
		H(\eta)-H(\eta\mid\xi)
		=
		H(\xi)+H(\eta)-H(\xi\vee\eta).
	\end{equation*}

	For $\nu\in \mathcal M(X)$, the \emph{$r$-neutralized rate-distortion function with fixed time and scale } is defined by
	\begin{align*}
		R_n(\nu,e^{-nr})
		=
		\inf
		\frac{1}{n}I(\xi,\eta),
	\end{align*}
	where the infimum is taken over all  random variables $\xi$ and $\eta=\{\eta_0,\dots,\eta_{n-1}\}$  defined on some probability space $(\Omega,\mathbb{P})$ such that
	\begin{itemize}
		\item $\xi$ takes values in $X$, and its law is given by $\nu$.
		\item Each $\eta_k$ takes values in $X$, and
		\begin{equation*}
			\mathbb{E}\left(\frac{1}{n}\sum_{k=0}^{n-1}d(f^k\xi,\eta_k)\right)<e^{-nr},
		\end{equation*}
	\end{itemize}

	Based on PS entropies \cite{PS07},
	we now localize the neutralized quantities  by using the   empirical measure.
	\begin{definition}
		For $\mu\in \mathcal M_f(X)$
		and $r>0$, we  define the \emph{localized $r$-neutralized rate-distortion quantity} and \emph{localized $r$-neutralized Brin-Katok entropy}  as
		\begin{align*}
			\mathcal R_\mu^r(f,X)
			=&
			\inf_{F\ni\mu}
			\limsup_{n\to\infty}
			\sup_{\substack{\nu\in \mathcal M(X)\\
					\nu(X_{n,F})=1}}
			R_n(\nu,e^{-nr}),\\
			\mathcal H_\mu^r(f,X)
			=&
			\inf_{F\ni\mu}
			\limsup_{n\to\infty}
			\sup_{\substack{\nu\in \mathcal M(X)\\
					\nu(X_{n,F})=1}}
			-\frac{1}{n} \int
			\log\nu(B_n(x,e^{-nr}))\,d\nu(x)
		\end{align*}
		where $F$ ranges over    neighborhoods of $\mu$ in
		$\mathcal M(X)$.
	\end{definition}

	We first give a commparision of
	these quantities.

	\begin{proposition}\label{prop 2.11}
		For any $r>0$ and $\nu\in \mathcal M_f(X)$, the three $r$-neutralized entropies are equal:
		$$\mathcal R_\mu^r(f,X)
		=\mathcal H_\mu^r(f,X)
		=\overline h_\mu^{PS,r}(f,X).$$
	\end{proposition}

	\begin{proof}
		We divide the proof into two steps.
		\begin{step1}
			We first show
			$\mathcal R_\mu^r(f,X)
			\leq \mathcal H_\mu^r(f,X)
			\leq \overline h_\mu^{PS,r}(f,X).$
		\end{step1}

		The first inequality follows from that
		\begin{align}\label{equ 2.3}
			R_n(\nu,e^{-nr})\leq -\frac{1}{n}\int
			\log\nu\bigl(B_n(x,e^{-nr})\bigr)\,\mathrm d\nu(x)
		\end{align}
		holds for any $\nu\in \mathcal M(X)$. Indeed, suppose the right-hand is finite and non-zero. Let $\xi: \Omega\to X$ be a random variable  with
		$\operatorname{Law}(\xi)=\nu$. Given $\xi=x$, we can choose a
		random variable $Z$ with conditional distribution
		\begin{align*}
			\mathbb P(Z\in A\mid\xi=x)
			=
			\frac{\nu(A\cap B_n(x,e^{-nr}))}
			{\nu(B_n(x,e^{-nr}))}.
		\end{align*}
		Define $	\eta=(Z,fZ,\ldots,f^{n-1}Z).$
		For $\nu$-almost every $x$,
		$\mathbb P(Z\in B_n(x,e^{-nr})\mid\xi=x)=1$.
		Hence, the   random variables  $(\xi,\eta)$ satisfying
		$\mathbb E\left[
		\frac{1}{n}\sum_{j=0}^{n-1}d(f^j\xi,\eta_j)
		\right]
		<e^{-nr}$.
		For a finite Borel partition $\mathcal Q=\{Q_1,\ldots,Q_N\}$,
		write
		\begin{align*}
			p_j(x)
			=\mathbb P(Z\in Q_j\mid\xi=x)
			=\frac{\nu(Q_j\cap B_n(x,e^{-nr}))}
			{\nu(B_n(x,e^{-nr}))}.
		\end{align*}
		Whenever $p_j(x)>0$,
		\begin{align*}
			\log\frac{p_j(x)}{\nu(Q_j)}
			\leq-\log\nu(B_n(x,e^{-nr})).
		\end{align*}
		The log-sum inequality and the definition of mutual
		information show
		\begin{align*}
			I(\xi;\eta)=I(\xi;Z)
			\leq\int_X\sum_{j=1}^N
			p_j(x)\log\frac{p_j(x)}{\nu(Q_j)}\,\mathrm d\nu(x)
			\leq-\int_X\log\nu(B_n(x,e^{-nr}))\,\mathrm d\nu(x),
		\end{align*}
		by taking the supremum over $\mathcal Q$.
		This proves \eqref{equ 2.3}, and hence
		$\mathcal R_\mu^r(f,X)\leq\mathcal H_\mu^r(f,X)$.

		We next show
		\begin{align}\label{equ 2.4}
			\mathcal H_\mu^r(f,X)
			\leq \overline h_\mu^{PS,r}(f,X)
		\end{align}
		Let $F$ be a neighborhood of $\mu$ and
		choose $q\geq1$ such that $2e^{-qr}<1$. The there exists   a
		neighborhood $F_0\subset F$ of $\mu$ such that
		$X_{n,F_0}\subseteq X_{n+q,F}$ for all sufficiently large $n$ since
		\begin{align*}
			\left|
			\int_X\varphi\,\mathrm d\mathcal E_{n+q}(x)
			-\int_X\varphi\,\mathrm d\mathcal E_n(x)
			\right|
			\leq
			\frac{2q\|\varphi\|_\infty}{n+q}
			\longrightarrow0.
		\end{align*}
		Let $B_1,\ldots,B_N$ be $n+q$-th Bowen balls of
		radius $e^{-(n+q)r}$ covering $X_{n+q,F}$ with $N=r_{n+q}(f,X_{n+q,F},e^{-(n+q)r}).$
		Since $X_{n,F_0}\subseteq X_{n+q,F}$, they also cover
		$X_{n,F_0}$. By the approach of Theorem \ref{thm 1.1}, we can construct a  Borel partition
		$\mathcal P$ of $X_{n,F_0}$ with at most $N$ atoms. Since $2e^{-qr}<1$, each atom has $d_n$-diameter
		less than $2e^{-(n+q)r}<e^{-nr}$. Thus, for every $\nu$ satisfying
		$\nu(X_{n,F_0})=1$,
		\begin{align*}
			-\frac{1}{n}\int_X
			\log\nu\bigl(B_n(x,e^{-nr})\bigr)\,\mathrm d\nu(x)
			\leq\frac{1}{n} H_\nu(\mathcal P)
			\leq\frac{1}{n}\log r_{n+q}(f,X_{n+q,F},e^{-(n+q)r}).
		\end{align*}
		Since $F$ was arbitrary and $\frac{n+q}{n}\to1$, we obtain
		\eqref{equ 2.4}.

		\begin{step2}
			We then show
			$	\mathcal R_\mu^r(f,X)
			\geq \overline h_\mu^{PS,r}(f,X).$
		\end{step2}
		Fix  an  open
		neighborhood $F$ of $\mu$ and put $m_n=n+\lceil\sqrt n\rceil$
		for any $n\in \mathbb N$.
		There exists an open neighborhood $F_0$ of $\mu$ such
		that $X_{n,F_0}\subseteq X_{m_n,F}$
		for all sufficiently large $n$.
		Suppose that $X_{n,F_0}$ is nonempty. Choose a maximal
		$(n,e^{-nr})$-separated set $E_n=\{x_1,\cdots,x_N\}$ of $X_{n,F_0}$, which implies
		$	|E_n|
		\geq
		r_n(f,X_{n,F_0},e^{-nr}).$
		Let $\nu_n$ be the uniform probability measure on $E_n$. Then $\nu_n(X_{m_n,F})=1.$
		Let $\xi$ be uniformly distributed on $E_n$, and let
		$\eta=(\eta_0,\ldots,\eta_{m_n-1})\in X^{m_n}$
		be random variables satisfying
		\begin{align*}
			\mathbb E\left[
			\frac1{m_n}\sum_{j=0}^{m_n-1}
			d(f^j\xi,\eta_j)
			\right]
			<
			e^{-rm_n}.
		\end{align*}

		Define a new random variable $\widehat{\xi}:\Omega\to E_n$ with
		$\widehat\xi=x_k$, where $k$ is the smallest index for which
		$\max_{0\leq j<n}d(f^jx_k,\eta_j)$ is minimal among all
		$x_1,\ldots,x_N$.
		If $\sum_{j=0}^{n-1}d(f^j\xi,\eta_j)<\frac12e^{-nr}$, then
		the definition of $\widehat\xi$ gives
		\begin{align*}
			d_n(\xi,\widehat\xi)
			&\leq
			\max_{0\leq j<n}d(f^j\xi,\eta_j)
			+\max_{0\leq j<n}d(f^j\widehat\xi,\eta_j)
			<e^{-nr}.
		\end{align*}
		Thus $\widehat\xi=\xi$.  Denote by $	P_e=\mathbb P(\widehat\xi\neq\xi).$
		Then by
		Markov's inequality, we obatin
		\begin{align*}
			P_e
			\leq
			\mathbb P\left(
			\sum_{j=0}^{n-1}d(f^j\xi,\eta_j)
			\geq\frac{1}{2}e^{-nr}
			\right)
			\leq
			2e^{nr}\,
			\mathbb E\left[\sum_{j=0}^{n-1}d(f^j\xi,\eta_j)\right]
			\leq
			2m_ne^{-r(m_n-n)}
			=:q_n.
		\end{align*}
		where the right-hand tends to $0$ as	$n\to\infty$. Since   $\xi$ is uniformly distributed on $E_n$, by  the
		data-processing inequality and Fano's inequality \cite{LT18}, we obtain
		\begin{align*}
			I(\xi;\eta)
			\geq
			I(\xi;\widehat\xi)
			\geq
			(1-P_e)\log|E_n|-H(P_e)\geq
			(1-P_e)\log|E_n|-\log 2
		\end{align*}
		for all sufficiently large $n$. It implies
		\begin{align*}
			R_{m_n}(\nu_n,e^{-rm_n})
			\geq
			\frac{1-q_n}{m_n}
			\log r_n(f,X_{n,F_0},e^{-nr})
			-\frac{\log2}{m_n}.
		\end{align*}
		Moreover, $\nu_n(X_{m_n,F})=1$. Therefore
		\begin{align*}
			\limsup_{n\to\infty}\sup_{\nu(X_{m_n,F}=1)}	R_{m_n}(\nu,e^{-rm_n})
			\geq \limsup_{n\to\infty}
			\left(\frac{1-q_n}{m_n}
			\log r_n(f,X_{n,F_0},e^{-nr})
			-\frac{\log2}{m_n}\right),
		\end{align*}
		which shows $\mathcal R_\mu^r(f,X)
		\geq \overline h_\mu^{PS,r}(f,X)$.
	\end{proof}

	We can also establish the relation between $r$-neutralized local entropy and PS entropy.

	\begin{proposition}\label{prop 2.12}
		For every $\mu\in \mathcal E_f(X)$,
		$	\underline h_\mu^{BK,r}(f) \leq
		\underline h_\mu^{PS,r}(f,X).$
	\end{proposition}

	\begin{proof}
	Suppose $\underline h_\mu^{BK,r}(f)>0$ and fix
	$0<a<\underline h_\mu^{BK,r}(f)$. Then for some $N\geq1$, the set
	\begin{align}\label{equ 2.5}
		G=\left\{x\in X:
		\mu(B_n(x,e^{-nr}))\leq e^{-an}
		\text{ for every }n\geq N\right\}
	\end{align}
	satisfies $\mu(G)=c>0$.
		Choose a positive integer $q$ such that $2e^{-qr}<1$ and put
		$m=n+q$. For every neighborhood $F$ of $\mu$, the ergodic theorem
		shows that
		$\mu(G\cap X_{m,F})>\frac{c}{2}$
		for all sufficiently large $n$.
		Let $\{B_m(x_i,e^{-mr})\}_{i=1}^K$ be a cover of $X_{m,F}$. For every ball intersecting $G\cap X_{m,F}$, choose
		$	y_i\in B_m(x_i,e^{-mr})\cap G\cap X_{m,F}.$
		By the triangle inequality,
		\begin{align*}
			B_m(x_i,e^{-mr})\cap(G\cap X_{m,F})
			\subseteq
			B_m(y_i,2e^{-mr})\subseteq
			B_m(y_i,e^{-nr}).
		\end{align*}
		Then
		by \eqref{equ 2.5}, one has
		$\frac{c}{2}
		<
		\mu(G\cap X_{m,F})
		\leq
		K e^{-an},$
		and hence $r_m(f,X_{m,F},e^{-mr})
		\geq
		\frac{c}{2}e^{an}.$
		Finally,
		$\underline h_\mu^{PS,r}(f,X)\geq a.$
		Letting $a\to\underline h_\mu^{BK,r}(f)$ completes the proof.
	\end{proof}

	\begin{remark}
		\begin{itemize}
			\item One can also show that
			$\overline h_\mu^{BK,r}(f) \leq
			\overline h_\mu^{PS,r}(f,X)$ for any $\mu\in \mathcal E_f(X)$.

			\item As shown in Example~\ref{exam 4.5}, the ergodicity in
			Proposition~\ref{prop 2.12}\textup{(2)} is essential.
		\end{itemize}
	\end{remark}

	\subsection{Variational principles  for the upper
		$r$-neutralized entropy}
	We first observe the following fact.
	\begin{lemma}\label{lem 2.14}
		If $U$ is an open neighborhood of $\mathcal M_f(X)$ in $\mathcal M(X)$, then
		$\mathcal E_n(X)\subseteq U$ for all sufficiently large $n$.
	\end{lemma}

	\begin{proof}
		For every $\varphi\in C(X)$, one has
		\begin{align*}
			\int\varphi\circ f\,\mathrm d\mathcal E_n(x)-\int\varphi\,\mathrm d\mathcal E_n(x)
			=\frac{\varphi(f^nx)-\varphi(x)}n,
		\end{align*}
		which tends to zero uniformly in $x$.  Suppose the conclusion fails, then there  exist $n_k\to\infty$ and $x_k\in X$ such that $\mathcal E_{n_k}(x_k)\notin U$. Passing to a subsequence, assume that $\mathcal E_{n_k}(x_k)\to\mu$. The displayed identity implies that $\mu$ is $f$-invariant. Since  $\mathcal M(X)\setminus U$ gives $\mu\notin U$, this makes a contradiction.
	\end{proof}

	We proceed to prove Theorem~\ref{thm 1.2}.

	\begin{proof}[Proof of Theorem~\ref{thm 1.2}]
		By Proposition~\ref{prop 2.11} and Lemma \ref{lem 2.9}, we only need to prove the variational principle:
		\begin{align*}
			\max_{\mu\in \mathcal M_f(X)}
			\overline h_\mu^{PS,r}(f,X)
			=
			\overline h_{top}^r(f,X).
		\end{align*}
		Since $X_{n,F}\subset X$, we only need to prove the upper bound.
		Suppose   that for each $\mu\in \mathcal M_f(X)$, $\overline h_\mu^{PS,r}(f,X)
		<
		\overline h_{top}^r(f,X).$
		Then there exists  $c\in\mathbb R$ such that
		\begin{align}\label{equ 2.6}
			\max_{\mu\in \mathcal M_f(X)}
			\overline h_\mu^{PS,r}(f,X)
			<c<
			\overline h_{top}^r(f,X).
		\end{align}
		For each $\mu\in \mathcal M_f(X)$, one can find an open neighborhood
		$F_\mu$ of $\mu$ such that
		\begin{align*}
			\limsup_{n\to\infty}
			\frac{1}{n}
			\log r_n(f,X_{n,F_\mu},e^{-nr})
			<c.
		\end{align*}
		By the compactness of $\mathcal M_f(X)$, there exist
		$\mu_1,\ldots,\mu_m\in \mathcal M_f(X)$ such that
		\begin{align*}
			\mathcal M_f(X)
			\subset
			\bigcup_{i=1}^mF_{\mu_i}
			=:U.
		\end{align*}
		By Lemma~\ref{lem 2.14}, for all sufficiently large $n$,
		$\mathcal E_n(X)\subset U.$
		Equivalently, $	X
		=
		\bigcup_{i=1}^mX_{n,F_{\mu_i}}.$
		Consequently,
		\begin{align*}
			r_n(f,X,e^{-nr})
			\leq
			\sum_{i=1}^m
			r_n(f,X_{n,F_{\mu_i}},e^{-nr}).
		\end{align*}
		It follows that
		\begin{align*}
			\overline h_{top}^r(f,X)
			\leq
			\max_{1\leq i\leq m}
			\limsup_{n\to\infty}
			\frac{1}{n}
			\log r_n(f,X_{n,F_{\mu_i}},e^{-nr})<c,
		\end{align*}
		which contradicts with \eqref{equ 2.6}. Therefore,
		\begin{align*}
			\overline h_{top}^r(f,X)
			=
			\max_{\mu\in \mathcal M_f(X)}
			\overline h_\mu^{PS,r}(f,X)
			=
			\max_{\mu\in \mathcal M_f(X)}
			\mathcal R_\mu^r(f,X).
		\end{align*}
	\end{proof}

	\begin{remark}
		Following the work of \cite{YCZ25}, we consider the family
		\begin{align*}
			\mathcal D
			=
			\left\{
			h_\nu^K(f,e^{-nr},\delta),\;
			h_\nu^P(f,e^{-nr}),\;
			h_\nu^S(f,\nu,e^{-nr},\delta),R_{n}(\nu,e^{-nr}),
			R_{n,\infty}(\nu,e^{-nr})
			\right\},
		\end{align*}
		where they denote the $r$-neutralized Katok entropy, partition entropy,  open cover quantities and modified rate-distortion quantity with the finite time $n$ and scale $e^{-nr}$,
		respectively. For each $Q\in \mathcal D$, the variational principles in
	 Theorem~\ref{thm 1.1} and Theorem \ref{thm 1.2} also hold.
		Moreover, its empirical-measure localization of $Q$  provides alternative
		characterizations of $r$-neutralized PS entropy.
		We will give their precise definitions and proofs
		in a separate paper.
	\end{remark}

	The maximum and the limit $r\to0$ can be interchanged.

	\begin{corollary}\label{cor 2.15}
		Let $(X,d)$ be a compact metric space and $f$ be a homeomorphism. Then
		\begin{align*}
			\lim_{r\to0}\overline h_{top}^r(f,X)
			=
			\lim_{r\to0}
			\max_{\mu\in \mathcal M_f(X)}
			\overline h_\mu^{PS,r}(f,X)
			=
			\max_{\mu\in \mathcal M_f(X)}
			\lim_{r\to0}
			\overline h_\mu^{PS,r}(f,X).
		\end{align*}
		The same conclusion holds with
		$\overline h_\mu^{PS,r}(f,X)$ replaced by
		$\mathcal R_\mu^r(f,X)$.
	\end{corollary}

	\begin{proof}
		By monotonicity of $r$, we only need to prove
		\begin{align*}
			\lim_{r\to0}
			\max_{\mu\in \mathcal M_f(X)}
			\overline h_\mu^{PS,r}(f,X)
			\leq
			\max_{\mu\in \mathcal M_f(X)}
			\lim_{r\to0}
			\overline h_\mu^{PS,r}(f,X).
		\end{align*}
		Let $\{r_k\}$ be a sequence decreasing to $0$, and set
		$h_k(\mu)=\overline h_\mu^{PS,r_k}(f,X)$.
		By Theorem~\ref{thm 1.2}, each $h_k$ is upper semicontinuous
		and attains its maximum on $\mathcal M_f(X)$. Choose $\mu_k$ such that
		\begin{align*}
			h_k(\mu_k)=\max_{\mu\in \mathcal M_f(X)}h_k(\mu).
		\end{align*}
		Then we can assume  $\mu_k\to\mu_*$ by passing to a subsequence (still denoted by $\{k\}_{k=1}^{\infty}$).
		Since $h_k$ decreases pointwise, denote by
		$L:=\lim_{k\to\infty}\max_\mu h_k(\mu)$.

		For every fixed $j$ and all $k\geq j$, one has $h_j\geq h_k$.
		By the upper semicontinuity, we have
		\begin{align*}
			h_j(\mu_*)
			\geq
			\limsup_{k\to\infty}h_j(\mu_k)
			\geq
			\lim_{k\to\infty}h_k(\mu_k)
			=L.
		\end{align*}
		Hence $L
		\leq
		\inf_{j\geq1}h_j(\mu_*)
		=
		\lim_{r\to0}
		\overline h_{\mu_*}^{PS,r}(f,X),$
		which proves the required inequality.
	\end{proof}

	It is natural to consider whether the corresponding variational principle holds for the lower $r$-neutralized entropies. The following example shows that it fails in general

	Consider the integer sequence $\{N_j\}$ with $N_1=1$ and $N_{j+1}=2^{N_j}$.  Set $I_j=\{N_j,\ldots,N_{j+1}-1\}$. Define the following set:
	\begin{align*}
		A&=\bigcup_{\substack{j\geq1\\j\ \mathrm{odd}}}I_j,
		\quad
		B=\bigcup_{\substack{j\geq1\\j\ \mathrm{even}}}I_j,
	\end{align*}
	and for $C=A,B$, define $	K_C=\{0\}\cup\left\{\frac{1}{n}:n\in C\right\}.$

	\begin{example}\label{exam 2.16}
		Let $\mathbb S=\mathbb R/8\mathbb Z$ and
		$\mathbb T=\mathbb R/\mathbb Z$,  equipped with their
		geodesic metrics. Fix an irrational $\alpha\in\mathbb T$ and take $f(u,v)=(u,v+\alpha)$. Put
		\begin{align*}
			X=\bigl(K_A\cup(2+K_B)\bigr)\times\mathbb T
			.
		\end{align*}
		Then $(X,f)$ is a compact invariant subsystem of the
		$C^\infty$ diffeomorphism
		of $\mathbb S\times\mathbb T$.
		For every $r>0$, we have
		\begin{align*}
			\max_{\mu\in\mathcal M_f(X)}
			\underline h_\mu^{PS,r}(f,X)
			=r
			<
			\underline h_{top}^r(f,X)
			=\frac{3r}{2}.
		\end{align*}
		Moreover, if $m_{\mathbb T}$ is the Haar measure on $\mathbb T$, for
		the ergodic measure $\mu_0=\delta_0\otimes m_{\mathbb T}$, we have
		\begin{align*}
			\underline h_{\mu_0}^{PS,r}(f,X)
			=r
			<
			\overline h_{\mu_0}^{PS,r}(f,X)
			=\frac{3r}{2}.
		\end{align*}
	\end{example}

	\begin{proof}
		We equip $X$ with the maximum product metric. Then $f$ is an isometry and thus by \eqref{equ 2.2}, one has
		\begin{align}\label{equ 2.7}
			\overline h_{top}^r(f,X)= r \overline{\dim}_B(X,d),\quad \underline h_{top}^r(f,X)= r \underline{\dim}_B(X,d)
		\end{align}

		We first consider $K_A$, which is a subset of $\{0\}\cup\left\{\frac{1}{n}:n\in N\right\}$. Then $\overline{\dim}_B(K_A)\leq \frac{1}{2}$. On the one hand,
		if $I_j\subset A$, the points $\{\frac{1}{k}:N_j\le k<2N_j\}$ are $\frac{1}{4N_j^2}$-separated, hence $	N(K_A,(4N_j^2)^{-1})\ge N_j.$ It follows $\overline{\dim}_B K_A=\frac{1}{2}.$
		On the other hand,
		if $I_j\cap A=\emptyset$, then one can show
		$	N(K_A,N_{j+1}^{-1})\le N_j+1.$ Thus $\underline{\dim}_B K_A=0.$
		The same dimension estimates hold for $K_B$.
		Since the box dimension of $K_A\cup K_B=\{0\}\cup\{1/n:n\in\mathbb N\}$ equals $\frac{1}{2}$, one has
		\begin{align*}
			\underline{\dim}_B(K_A\cup(2+K_B))
			=\overline{\dim}_B(K_A\cup(2+K_B))
			=\frac{1}{2}.
		\end{align*}
		By the product inequalities \cite{Pes97}, we have $\underline{\dim}_B X
		=\overline{\dim}_B X
		=\frac{3}{2}$.
		Hence by \eqref{equ 2.7},  $$\underline h_{top}^r(f,X)
		=\overline h_{top}^r(f,X)=\frac{3r}{2}.$$

		Every invariant measure on $X$ is of the form
		$\mu=\rho\otimes m_{\mathbb T}$, where
		$\rho\in\mathcal M(K_A\cup(2+K_B))$. Then $\mathcal E_n(u,v)$   converges uniformly to $\delta_u\otimes m_{\mathbb T}$ on $X$.
		If $\rho$ is not Dirac, then one can choose an open neighborhood $F\ni\mu$ whose closure is disjoint from
		$\{\delta_u\otimes m_{\mathbb T}:u\in K_A\cup(2+K_B)\}$.
		Then $\underline h_\mu^{PS,r}(f,X)=0$.

		Suppose $\mu=\delta_u\otimes m_{\mathbb T}$.
		For $u\in K_A$, we take the neighborhood
		$$F=\{\nu:\nu(K_A\times\mathbb T)>\frac{1}{2}\}.$$
		Then $X_{n,F}=K_A\times\mathbb T$  and thus $
		\underline h_\mu^{PS,r}(f,X)\le r\underline{\dim}_B(K_A\times\mathbb T)=r.$
		For any neighborhood $F\ni\mu$, by the uniform convergence, one has
		$\{u\}\times\mathbb T\subset X_{n,F}$ and hence
		$\underline h_\mu^{PS,r}(f,X)\geq r$.
		Thus $\underline h_\mu^{PS,r}(f,X)=r$.
		The argument also holds for $u\in2+K_B$. Therefore
		$$
		\max_{\mu\in\mathcal M_f(X)}
		\underline h_\mu^{PS,r}(f,X)=r<\underline h_{top}^r(f,X).
		$$

		For the ergodic measure $\mu_0=\delta_0\otimes m_{\mathbb T}$, we are left to show $\overline h_{top}^r(f,X)=\frac{3r}{2}$.
		Fix  an open neighborhood $F\ni\mu_0$. Then there exists  $\gamma>0$ such that
		$(K_A\cap[0,\gamma])\times\mathbb T
		\subset X_{n,F}$
		for all large $n$.
		Since $K_A\setminus[0,\gamma]$ is finite, by \cite{Pes97}, one can show
		$$
		\limsup_{n\to\infty}\frac1n
		\log s_n(f,X_{n,F},e^{-nr})
		\ge r\overline{\dim}_B
		\bigl((K_A\cap[0,\gamma])\times\mathbb T\bigr)
		=\frac{3r}{2}.
		$$
		Hence, we have $\underline h_{\mu_0}^{PS,r}(f,X)
		=r
		<
		\overline h_{\mu_0}^{PS,r}(f,X)
		=\frac{3r}{2},$ which finishes the proof.
	\end{proof}

	\section{Dimension estimates for $r$-neutralized Pfister--Sullivan entropy on the whole manifold}
	\label{sec 3}
	We first give the estimates of the $r$-neutralized PS entropies on compact invariant sets and then on the whole
	manifold and then obtain exact formula  for $C^\infty$ and Anosov diffeomorphisms.

	\subsection{Upper bounds}

	Throughout this section, let $f:M\to M$ be a $C^1$
	diffeomorphism of a closed Riemannian $m$-manifold $M$ (compact and without
	boundary), equipped with its Riemannian distance $d$.

	Let $\mu\in\mathcal M_f(M)$.
	A point $x\in M$ is \emph{regular} if its tangent space admits
	a splitting $T_xM=E_1(x)\oplus\cdots\oplus E_{\ell(x)}(x)$ and
	distinct numbers $\lambda_1(x)>\cdots>\lambda_{\ell(x)}(x)$
	such that, for every $0\neq v\in E_i(x)$,
	\begin{align*}
		\lim_{n \to \pm\infty} \frac{1}{n} \log \| Df_x^n v \| = \lambda_i(x).
	\end{align*}
	The numbers $\lambda_i(x)$ are the Lyapunov exponents at $x$, with
	multiplicities $d_i(x)=\dim E_i(x)$. Denote by $\Gamma\subset M$ the set of regular points. Then
	by the multiplicative ergodic
	theorem \cite{Via14}, $\Gamma$ has full $\mu$-measure, and
	these exponents and multiplicities are measurable and $f$-invariant. If $\mu$ is ergodic, the exponents and their multiplicities are
	constant $\mu$-a.e.
	At each regular point $x\in \Gamma$, denote by $E^s(x)$, $E^c(x)$ and $E^u(x)$  the
	sums of the spaces $E_i(x)$ with negative, zero and positive
	exponents, respectively. Then
	\begin{align*}
		T_xM=E^s(x)\oplus E^c(x)\oplus E^u(x).
	\end{align*}

	Pesin and Ruelle studied the relation between measure-theoretic entropy and Lyapunov exponents.
	\begin{itemize}
		\item \textup{Ruelle's inequality \cite{Rue78}:}
		$	h_\mu(f)\leq
		\min\{\mathcal X^+(f,\mu),\mathcal X^-(f,\mu)\}.$

		\item \textup{Pesin's entropy formula \cite{LS82,Pes77}:}
		if $f$ is $C^{1+\alpha}$ and $\mu$ is an SRB measure, then
		\begin{align*}
			h_\mu(f)=\mathcal X^+(f,\mu)
			=\min\{\mathcal X^+(f,\mu),\mathcal X^-(f,\mu)\}.
		\end{align*}
	\end{itemize}
	If we restrict $\mu$ to be an invariant measure of the compact $f$-invariant set $\Lambda$, the above results also hold.

	We introduce the  notations for exterior powers and singular values
	\cite{Via14,BP07}. For $q\geq1$ and $0\leq k\leq m$, the
	derivative $D_xf^q$ induces the linear map
	$	\wedge^kD_xf^q:
	\wedge^kT_xM
	\longrightarrow
	\wedge^kT_{f^qx}M.$
	Its norm $\|\wedge^kD_xf^q\|$ represents the maximal expansion
	factor of $k$-dimensional volume under $D_xf^q$. We use the
	convention $\|\wedge^0D_xf^q\|=1.$
	Let $\sigma_1(x)\geq\cdots\geq\sigma_m(x)>0$
	be the singular values of $D_xf^q$. Then for $1\leq k\leq m,$ $	\left\|\wedge^kD_xf^q\right\|
	=
	\sigma_1(x)\cdots\sigma_k(x).$
	Define the logarithm of the maximal volume expansion of $D_xf^q$
	among all dimensions by
	\begin{align*}
		\Psi_q(x)
		:=
		\log\max_{0\leq k\leq m}
		\left\|\wedge^kD_xf^q\right\|
		=
		\sum_{i=1}^m\max\{\log\sigma_i(x),0\}.
	\end{align*}

	We first establish the upper bound of $r$-neutralized PS entropy in terms of positive
	Lyapunov exponents.
	\begin{lemma}\label{lem 3.1}
		Let $f$ be a $C^1$ diffeomorphism of a closed Riemannian
		$m$-manifold. Then for every $\mu\in \mathcal M_f(M)$ and every $r>0$,
		\begin{align*}
			\overline h_\mu^{PS,r}(f,M)
			\leq mr+\mathcal X^+(f,\mu).
		\end{align*}
	\end{lemma}

	\begin{proof}For any set $S\subset M$, we denote by $\mathcal{N}_a(S)$ its
		$a$-neighborhood. Then we divide the proof into  two steps.
		\begin{step1}
			We prove that there
			exists a constant
			$C\geq 1$ such that for every
			$q\in \mathbb N$ and  sufficiently small
			$a>0$, there is a finite Borel partition
			$\mathcal P_a$ of $M$
			satisfying:
			\begin{itemize}
				\item $\operatorname{diam} \mathcal P_a
				<2a$;

				\item
				$\#\mathcal P_a\leq Ca^{-m}$;

				\item for every $Q\in\mathcal P_a$ ,
				\begin{align*}
					\#\{Q'\in\mathcal P_a:f^q(Q)\cap Q'\neq\emptyset\}
					\leq Ce^{\inf_{x\in Q}\Psi_q(x)}.
				\end{align*}
			\end{itemize}
			Moreover, for every  $\nu\in \mathcal M(M)$,
			\begin{align}
				H_\nu(f^{-q}\mathcal P_a\mid\mathcal P_a)
				\leq
				\log C+\int_M\Psi_q\,\mathrm d\nu.
				\label{equ 3.1}
			\end{align}
		\end{step1}

		Take a maximal $a$-separated set $\{x_1,\ldots,x_N\}$ in $M$
		and the Voronoi partition $\mathcal P_a=\{Q_1,\ldots,Q_N\}$, where
		\begin{align*}
			Q_i
			=
			\{x\in M:d(x,x_i)\leq d(x,x_j)\text{ for every }1\leq j\leq N\}
		\end{align*}
		with ties assigned to the smallest index.
		Each $Q_i$ lies in $B(x_i,a)$, so $\operatorname{diam}\mathcal P_a<2a$.
	The balls $B(x_i,a/2)$ are pairwise disjoint. Uniform
	volume bounds on the closed manifold $M$ therefore give
	$\#\mathcal P_a\leq C_1a^{-m},$
	where $C_1$ is independent of $a$.

		Fix $Q\in\mathcal P_a$ and $x\in Q$. Since $f^q$ is $C^1$ and
		$M$ is compact, for each fixed $q$ and all sufficiently small $a>0$, one has
		$f^q\bigl(B(x,2a)\bigr)
		\subseteq
		\mathcal N_a(\mathcal E_{x,a})$,
		where $\mathcal E_{x,a}
		:=
		\exp_{f^qx}\bigl(D_xf^q(B(0,2a))\bigr).$
		In normal coordinates at $f^qx$, the set $\mathcal E_{x,a}$
		corresponds to an ellipsoid with semi-axes
		$2a\sigma_1(x),\ldots,2a\sigma_m(x)$.
		If $Q'\in\mathcal P_a$ satisfies
		$f^q(Q)\cap Q'\neq\emptyset$, then its center $x_{Q'}$ belongs to
		$\mathcal N_{2a}(\mathcal E_{x,a})$, and hence
		\begin{align*}
			B\left(x_{Q'},\frac{a}{2}\right)
			\subseteq
			\mathcal N_{\frac{5a}{2}}(\mathcal E_{x,a}).
		\end{align*}
		The balls $B(x_{Q'},a/2)$ are pairwise disjoint.
		The volume of $\mathcal N_{5a/2}(\mathcal E_{x,a})$
		is bounded by
		$C_2a^m\prod_{i=1}^m(1+\sigma_i(x))$,
		whereas each such ball has volume bounded below
		by a fixed multiple of $a^m$. Hence
		\begin{align*}
			\#\{Q'\in\mathcal P_a:f^q(Q)\cap Q'\ne\varnothing\}
			\leq C_3\prod_{i=1}^m(1+\sigma_i(x))
			\leq Ce^{\Psi_q(x)}.
		\end{align*}
		The constants are independent of $q,a,Q$, and $x\in Q$.
		Together with $\#\mathcal P_a\leq C_1a^{-m}$,
		this gives the asserted partition bounds. It also yields
		\begin{align*}
			H_\nu(f^{-q}\mathcal P_a\mid\mathcal P_a)
			\leq
			\sum_{Q\in\mathcal P_a}
			\nu(Q)
			\left(
			\log C+\inf_{x\in Q}\Psi_q(x)
			\right)
			\leq
			\log C+\int_M\Psi_q\,\mathrm d\nu.
		\end{align*}

		\begin{step2}
		 For each
			positive integer $q$, we can obtain
			\begin{align*}
				\overline h_\mu^{PS,r}(f,M)
				\leq
				mr+\frac{\log C}{q}
				+\frac1q\int_M\Psi_q\,\mathrm d\mu.
			\end{align*}
		\end{step2}

		Since $M$ is compact and $f$ is $C^1$, we take
		\begin{align*}L
			=
			\max\left\{
			2,\,
			\max_{x\in M}\|D_xf\|,\,
			\max_{x\in M}\|D_xf^{-1}\|
			\right\}
		\end{align*}
		so that both $f$ and $f^{-1}$ are
		$L$-Lipschitz.
		For all sufficiently large $n$, set $a_n=\frac{e^{-nr}}{4L^q}$
		and let $\mathcal P_n$ be the partition constructed in Step~1.
		Fix $\eta>0$ and consider the convex open neighborhood
		\begin{align*}
			F
			=
			\left\{
			\nu\in \mathcal M(M):
			\int_M\Psi_q\,\mathrm d\nu
			<
			\int_M\Psi_q\,\mathrm d\mu+\eta
			\right\}
		\end{align*}
		of $\mu$. Let $E_n\subseteq M_{n,F}$ have maximal
		cardinality among $(n,e^{-nr})$-separated sets, and let
		$\nu_n$ be uniform on $E_n$. Define
		\begin{align}\label{equ 3.2} \overline{\nu}_n
			=
			\frac{1}{n}\sum_{t=0}^{n-1}f_*^t\nu_n
			=
			\frac1{\#E_n}\sum_{x\in E_n}\mathcal E_n(x)
		\end{align}
		Since $F$ is convex, one has $\overline{\nu}_n\in F$.

		Fix
		$0\leq j\leq q-1$ and denote by $p_j:=[\frac{n-1-j}{q}]$. Consider the partition $$\mathcal Q_{n,j}
		=\bigvee_{k=0}^{p_j}
		f^{-(j+kq)}\mathcal P_{n}.$$
		The partition $\mathcal Q_{n,j}$ separates $E_n$. Indeed, if
		$x,y\in E_n$ lie in the same atom, then $d(f^{j+kq}x,f^{j+kq}y)<2a_n$. For each $0\leq t\leq n-1$, the Lipschitz property indicates
		$d(f^tx,f^ty)
		<
		2L^qa_n
		<
		e^{-nr},$
		contrary to the separation of $E_n$. Hence
		$H_{\nu_n}(\mathcal Q_{n,j})=\log\#E_n$.
		On the other hand,
		\begin{align*}
			H_{\nu_n}(\mathcal Q_{n,j})\leq H_{\nu_n}(f^{-j}\mathcal P_{n})+\sum_{k=1}^{p_j} H_{\nu_n}\left(f^{-(j+kq)}\mathcal P_{n}\mid \bigvee_{l=0}^{k-1} f^{-(j+lq)}\mathcal P_n\right)
		\end{align*}
		Summing over $0\leq j<q$ and combining the preceding inequalities,
		we obtain
		\begin{align*}
			q\log\#E_n
			\leq
			q\log\#\mathcal P_n+
			\sum_{j=0}^{q-1}\sum_{k=1}^{p_j}
			H_{f_*^{j+(k-1)q}\nu_n}
			\bigl(f^{-q}\mathcal P_n\mid\mathcal P_n\bigr).
		\end{align*}
		By \eqref{equ 3.1} and $\Psi_q\geq0$,
		\begin{align*}
			q\log\#E_n
			\leq
			q\log\#\mathcal P_n
			+(n-q)\log C
			+\sum_{t=0}^{n-1}
			\int_M\Psi_q\,\mathrm d(f_*^t\nu_n).
		\end{align*}
		Dividing both sides by $nq$ and using \eqref{equ 3.2}, we obtain
		\begin{align*}
			\frac{1}{n}\log\#E_n
			&\leq
			mr+\frac{m\cdot \log 4L^q+\log C}{n}
			+\frac{\log C}{q}
			+\frac{1}{q}\int_M\Psi_q\,\mathrm d \overline{\nu}_n,
		\end{align*}
		Since $\overline{\nu}_n\in F$, taking the upper limit and letting $\eta\to0$ gives
		\begin{align*}
			\overline h_\mu^{PS,r}(f, M)
			\leq
			mr+\frac{\log C}{q}
			+\frac1q\int_M\Psi_q\,\mathrm d\mu.
		\end{align*}
		Since $0\leq\frac{\Psi_q(x)}q\leq m\log L$, the multiplicative
		ergodic theorem and dominated convergence give
		\begin{align*}
			\overline h_\mu^{PS,r}(f, M)
			\leq
			mr+	\lim_{q\to\infty}
			\frac1q\int_M\Psi_q\,\mathrm d\mu
			=mr+
			\mathcal X^+(f,\mu).
		\end{align*}
	\end{proof}

	If we further strengthen the smoothness assumptions on the system, we can obtain a stronger upper bound estimate. Before the result, we need the following observation.
	\begin{lemma}\cite{PS07}\label{lem 3.2}
		For every $\mu\in \mathcal M_f(M)$ and every $\delta>0$, we have
		\begin{equation*}
			\inf_{F\ni\mu}\limsup_{n\to\infty}\frac{1}{n}
			\log s_n(f,X_{n,F},\delta)
			\leq h_\mu(f).
		\end{equation*}
		If $\mu$ is ergodic, then for every $\eta>0$
		there is $\delta>0$
		such that  for each neighborhood $F\ni\mu$, there exists $N_F\in \mathbb N$ such that for any $n\geq N_F$,
		\begin{equation*}
			s_n(f,X_{n,F},4\delta)
			\geq e^{n(h_\mu(f)-\eta)}.
		\end{equation*}
	\end{lemma}

	The Reparametrization lemma  also plays an important role in the proof.

	\begin{lemma}\cite[Lemma~3.3]{BLY15}\label{lem 3.3}
		Let $\sigma:[0,1]^m\to\mathbb R^m$ be a $C^k$ $m$-disk with
		$\|\sigma\|_k<1$ and $(T_n:(-2,2)^m\to\mathbb R^m)_n$ be
		$C^k$ maps  with
		$\|D^{j+1}T_n\|\leq A_n$ for $1\leq j\leq k-1$, where
		$A_n\geq\max\{\|DT_n\|,1\}$. Then  for every $n$, there exists
		a family $\mathcal F_n:=\{\psi_n:[0,1]^m\to[0,1]^m\}$ of continuous maps
		, real analytic on $(0,1)^m$
		such that
		\begin{enumerate}
			\item[\textup{(i)}] $\|T^n\circ\sigma\circ\psi_n\|_k\leq1$;
			\item[\textup{(ii)}] $\|D(T^j\circ\sigma\circ\psi_n)\|\leq1$
			for $0\leq j\leq n$;
			\item[\textup{(iii)}]
			$\displaystyle\bigcap_{j=0}^{n}(T^j\circ\sigma)^{-1}((-1,1)^m)
			\subseteq\bigcup_{\psi_n\in\mathcal F_n}\psi_n([0,1]^m)$;
			\item[\textup{(iv)}]
			$\displaystyle\#\mathcal F_{n+1}
			\leq\widetilde C_{k,m}\#\mathcal F_n(A_nB_k^2)^{m/k}$.
		\end{enumerate}
		where $T^n=T_n\circ\cdots\circ T_1$, the constant $\widetilde C_{k,m}$ depends only on $k,m$, and $B_k$ is the $k$-th Bell number.
	\end{lemma}

	\begin{lemma}\label{lem 3.4}
		Let $f:M\to M$ be a  $C^k$ diffeomorphism with $k\geq 2$ of a closed
		Riemannian $m$-manifold.
		For any integer $p\geq1$, put $\psi_p(x)=\log\max\{1,\|D_xf^p\|\}.$ Then
		there exists $C_k\geq1$
		such that  for every  $\eta>0$,
		there exists $\delta$ and $C>0$ such that
		\begin{align*}
			r_n(f,B_n(x,\delta),\epsilon)
			\leq C\epsilon^{-m}
			\exp\biggl(
			n\eta+\frac{n\log C_k}{p}
			+\frac{m}{kp}
			\sum_{t=0}^{n-1}\psi_p(f^tx)
			\biggr)
		\end{align*}
		holds	for every $x\in M$, $n\geq p$ and $0<\epsilon<1$.
	\end{lemma}

	\begin{proof}
		Fix $\eta>0$. Consider normal coordinates along the orbit
		of $x$ and choose $0<r_M<\operatorname{inj}(M)$.
		Since $f^p$ is uniformly continuous, there exits $\gamma>0$
		such that for any $y,z\in M$,
		\begin{align*}
			d(y,z)<\gamma
			\quad\Longrightarrow\quad
			d(f^py,f^pz)<r_M.
		\end{align*}
		Choose
		$0<h<\frac{\min\{\gamma,r_M\}}{2\sqrt m}$.
		For $t\geq0$ and $u\in(-2,2)^m$, define
		\begin{align*}
			T_t(u)
			=h^{-1}\exp_{f^{t+p}x}^{-1}
			\bigl(f^p(\exp_{f^tx}(hu))\bigr).
		\end{align*}
		Thus $T_t$ is well-defined on $(-2,2)^m$ for all $x\in M$
		and $t\geq0$.
		Define
		\begin{align*}
			A_t(x)
			=\exp\left(\frac{pk\eta}{2m}\right)\cdot
			\max\{1,\|D_{f^tx}f^p\|\}.
		\end{align*}
		As $h\to 0$, $DT_t(u)\to D_{f^t x}f^p$ in operator norm, which is uniformly in $x,t,u$. Moreover, $\|D^jT_t\|_\infty=O(h^{j-1})$ for $2\le j\le k$.
		Then	by decreasing $h$ if necessary, for $1\leq j\leq k$,
		we obtain
		\begin{align}\label{equ 3.3}
			\|D^jT_t\|_\infty\leq A_t(x).
		\end{align}

		Take $\delta=\frac{h}{4}$ and fix $a\in\{0,\ldots,p-1\}$.  We stratify time modulo $p$. Define
		\begin{align*}
			N_a=\#\{j\geq0:a+(j+1)p<n\}.
		\end{align*}
		For $y\in B_n(x,\delta)$, we take $u_a(y)
		:=h^{-1}\exp_{f^ax}^{-1}(f^ay)$. Then
		$u_a(y)\in B(0,\frac{1}{4})$ and for
		$1\leq j\leq N_a$, we have the following relation:
		\begin{align*}
			T_{a+(j-1)p}\circ\cdots\circ T_a(u_a(y))
			=h^{-1}\exp_{f^{a+jp}x}^{-1}(f^{a+jp}y)
			\in B(0,\frac{1}{4})\subset (-1,1)^m.
		\end{align*}
		By the compactness, we can cover $B(0,\frac{1}{4})$ by finite affine
		$m$-disks
		$\sigma:[0,1]^m\to(-1,1)^m$ with
		$\|\sigma\|_k<1$. We denote by $C_0$ the amount of such $m$-disks.
		By \eqref{equ 3.3}, we may apply
		Lemma~\ref{lem 3.3} to each initial disk
		$\sigma$ and the sequence
		$T_a,T_{a+p},\ldots,T_{a+(N_a-1)p}$.
		Thus there are at most
		\begin{align}\label{equ 3.4}
			C_0\prod_{j=0}^{N_a-1}
			C_{k,m}A_{a+jp}(x)^{m/k}
		\end{align}
		maps of the form $\sigma\circ\varphi$ such that their images  cover
		$\{u_a(y):y\in B_n(x,\delta)\}$, where
		$\varphi:[0,1]^m\to[0,1]^m$.

		For each map $\sigma\circ\varphi$ ,  we define
		\begin{align*}
			\Phi(v)
			=f^{-a}\bigl(\exp_{f^ax}(h\sigma(\varphi(v)))\bigr).
		\end{align*}
		Put $\rho_0=\sigma\circ\varphi$ and for $1\leq j\leq N_a$, set
		\begin{align*}
			\rho_j
			=T_{a+(j-1)p}\circ\cdots\circ T_a\circ\rho_0.
		\end{align*}
		By Lemma~\ref{lem 3.3}\textup{(ii)}, one has
		$\|D\rho_j\|_\infty\leq1$.
		Since $f^{a+jp}\circ\Phi
		=\exp_{f^{a+jp}x}\circ(h\rho_j)$,
		the derivatives of these maps are uniformly bounded.
		As $p$ is fixed, the same holds at the intervening times.
		Hence, there exists $L_p>0$ (independent of $x,n$) such that
		\begin{align*}
			d_n(\Phi(v),\Phi(w))\leq L_p|v-w|.
		\end{align*}

		Notice that we can divide $[0,1]^m$ into  cubes $Q$ of diameter less than $\frac{\epsilon}{2L_p}$
		with the amount at most
		$C_p\epsilon^{-m}$, where $C_p$ only depends on $m$ and $p, L_p$. For each cube $Q$ with
		$\Phi(Q)\cap B_n(x,\delta)\ne\varnothing$,
		one can choose $z_Q\in\Phi(Q)\cap B_n(x,\delta)$.
		If $y\in\Phi(Q)\cap B_n(x,\delta)$, then
		\begin{align*}
			d_n(y,z_Q)\leq L_p\operatorname{diam}(Q)
			<\epsilon.
		\end{align*}
		Thus, using \eqref{equ 3.4}, we have
		\begin{align*}
			r_n(f,B_n(x,\delta),\epsilon)
			\leq C_pC_0\epsilon^{-m}
			\prod_{j=0}^{N_a-1}
			C_{k,m}A_{a+jp}(x)^{\frac{m}{k}}
			\leq C'_p\epsilon^{-m}
			\prod_{j=0}^{N_a-1}
			C_{k,m}A_{a+jp}(x)^{\frac{m}{k}},
		\end{align*}
		where $C'_p=C_pC_0$.
	Taking the geometric mean of these $p$ bounds and using
	$\{a+jp:0\leq a<p,\ 0\leq j<N_a\}
	=\{0,\ldots,n-p-1\}$, we obtain
	\begin{align*}
		r_n(f,B_n(x,\delta),\epsilon)
		&\leq C'_p\epsilon^{-m}
		C_{k,m}^{(n-p)/p}
		\left(\prod_{t=0}^{n-p-1}A_t(x)\right)^{m/(kp)}\\
		&\leq C'_p\epsilon^{-m}
		\exp\biggl(
		n\eta+\frac{n\log C_{k,m}}p
		+\frac{m}{kp}\sum_{t=0}^{n-1}\psi_p(f^tx)
		\biggr).
	\end{align*}
	This proves the lemma.
	\end{proof}

	\begin{proposition}\label{prop 3.5}
		Let $f:M\to M$ be a  $C^k$ diffeomorphism with $k\geq 2$ of a closed
		Riemannian $m$-manifold. For every
		$\mu\in\mathcal M_f(M)$ and $r>0$,
		\begin{align*}
			\overline h_\mu^{PS,r}(f,M)
			\leq mr+\min\left\{
			\mathcal X^+(f,\mu),
			h_\mu(f)+\frac{m}{k}\chi_{\max}^+(f,\mu)
			\right\}.
		\end{align*}
	\end{proposition}

	\begin{proof}
		By Lemma~\ref{lem 3.1}, it suffices to prove
		\begin{align}\label{equ 3.5}
			\overline h_\mu^{PS,r}(f,M)\leq mr+h_\mu(f)+\frac{m}{k}\chi_{\max}^+(f,\mu).
		\end{align}
		Fix $p\geq1$ and $\eta>0$, and let $\delta>0$ be given
		by Lemma~\ref{lem 3.4}. Choose a neighborhood $F_1$ of
		$\mu$ such that, for every $\nu\in F_1$,
		$\int_M\psi_p\,\mathrm d\nu
		\leq \int_M\psi_p\,\mathrm d\mu+\eta.$
		By Lemma~\ref{lem 3.2}, there is a neighborhood
		$F_2$ of $\mu$ such that
		\begin{align*}
			\limsup_{n\to\infty}\frac1n
			\log s_n(f,M_{n,F_2},\delta)
			\leq h_\mu(f)+\eta.
		\end{align*}
		Set $F=F_1\cap F_2$, and let $E_n$ be a maximal $(n,\delta)$-separated subset of $M_{n,F}$. Then one has
		\begin{align*}
			M_{n,F}\subseteq\bigcup_{x\in E_n}B_n(x,\delta).
		\end{align*}
		For every $x\in E_n$, we have $\mathcal E_n(x)\in F_1$, and hence
		\begin{align*}
			\frac{1}{n}\sum_{t=0}^{n-1}\psi_p(f^tx)
			=\int_M\psi_p\,\mathrm d\mathcal E_n(x)
			\leq\int_M\psi_p\,\mathrm d\mu+\eta.
		\end{align*}
		Apply Lemma  \ref{lem 3.4} with
		$\epsilon=\frac{1}{2}e^{-nr}$ to each $B_n(x,\delta)$,
		$x\in E_n$.  Consequently,
		\begin{align*}
			r_n(f,M_{n,F},e^{-nr})
			\leq C2^m e^{mrn}\#E_n
			\exp\left(
			n\eta+\frac{n\log C_k}{p}
			+\frac{mn}{kp}
			\left(\int_M\psi_p\,\mathrm d\mu+\eta\right)
			\right).
		\end{align*}
		Since $\#E_n\leq s_n(f,M_{n,F_2},\delta)$, Lemma~\ref{lem 3.2}
		gives
		\begin{align*}
			\overline h_\mu^{PS,r}(f,M)
			\leq h_\mu(f)+mr+\frac{\log C_k}{p}
			+\frac{m}{sp}\int_M\psi_p\,\mathrm d\mu
			+\left(2+\frac{m}{kp}\right)\eta.
		\end{align*}
		Letting $\eta\to0$ shows
		\begin{align*}
			\overline h_\mu^{PS,r}(f,M)
			\leq h_\mu(f)+mr+\frac{\log C_k}{p}
			+\frac{m}{kp}\int_M\psi_p\,\mathrm d\mu.
		\end{align*}
		By the multiplicative ergodic theorem, for $\mu$-a.e. $x\in M$, one has $\frac{\psi_p(x)}p
		\longrightarrow\max\{\lambda_1(x),0\}.$
		Moreover, $0\leq\frac{\psi_p(x)}p
		\leq\log\max\{1,\|Df\|_\infty\}.$
		Dominated convergence yields
		\begin{align*}
			\lim_{p\to\infty}\frac1p
			\int_M\psi_p\,\mathrm d\mu
			=\chi_{\max}^+(f,\mu).
		\end{align*}
		Letting $p\to\infty$ proves
		\eqref{equ 3.5}.
	\end{proof}

	We first give the proof of Theorem \ref{thm 1.3}.
	\begin{proof}[Proof of
		Theorem~\ref{thm 1.3}]
		Since $\mathcal X^+(f^{-1},\mu)=\mathcal X^-(f,\mu)$ and  $\mathcal X^+_{\max}(f^{-1},\mu)=\mathcal X^-_{\max}(f,\mu)$ by Lemma \ref{lem 3.1} and Proposition \ref{prop 3.5}, we are left to prove
		\begin{align*}
			\overline h_\mu^{PS,r}(f, M)=\overline h_\mu^{PS,r}(f^{-1}, M)
		\end{align*}
		For every $n\geq1$, the map  $\Phi_n(x)=f^{n-1}x$ is an isometry between the dynamical systems $( M,f,d_n)$ and $( M,f^{-1},d_n)$.
		Moreover, one can observe that
		$\mathcal E_n\bigl(f^{-1},\Phi_n(x)\bigr)
		=
		\mathcal E_n(f,x).$
		Thus, for every neighborhood $F$ of $\mu$, the map $\Phi_n$
		maps $ M_{n,F}^f$ isometrically onto
		$ M_{n,F}^{f^{-1}}$, and thus
		\begin{align*}
			r_n\bigl(f, M_{n,F}^f,e^{-nr}\bigr)
			=
			r_n\bigl(f^{-1}, M_{n,F}^{f^{-1}},e^{-nr}\bigr).
		\end{align*}
		Then we obtain the desired result.
	\end{proof}

	The following example shows that
	$\min\left\{
	\Phi_1^*(f,\mu),
	h_\mu(f)+\frac{m}{k}\Phi_2^*(f,\mu)
	\right\}$
	can be strictly smaller than $\Phi_1^*(f,\mu)$.

	\begin{example}\label{exam 3.6}
		Let $f_A:\mathbb T^2\to\mathbb T^2$ be the toral automorphism
		induced by
		\begin{align*}
			A=\begin{pmatrix}2&1\\1&1\end{pmatrix}.
		\end{align*}
		Then $f_A$ is a $C^\infty$ diffeomorphism. Take $\lambda=\log\frac{3+\sqrt5}{2}$ and fix  $k\in \mathbb N$ and $k> 2$.
		Consider the following two ergodic measures:
		\begin{itemize}
			\item Let $\mu_{\mathrm H}$ be Haar measure. Then its  Lyapunov exponents are $\lambda$ and $-\lambda$. Since
			$h_{\mu_{\mathrm H}}(f_A)=\lambda$, it shows
			\begin{align*}
				\Phi_1^*(f_A,\mu_{\mathrm H})=\lambda<\lambda+\frac{2\lambda}{k}=h_{\mu_{\mathrm H}}(f_A)
				+\frac{2}{k}\Phi_2^*(f_A,\mu_{\mathrm H}).
			\end{align*}

			\item Let $\mu_0=\delta_0$. Then     Lyapunov exponents are also
			$\lambda$ and $-\lambda$  while $h_{\mu_0}(f_A)=0$. Hence
			\begin{align*}
				h_{\mu_0}(f_A)
				+\frac{2}{k}\Phi_2^*(f_A,\mu_0)
				=\frac{2\lambda}{k}<\lambda=\Phi_1^*(f_A,\mu_0).
			\end{align*}
		\end{itemize}
	\end{example}

	\subsection{Lower bounds}
	We recall Lyapunov charts from \cite{DQ25,LY85a,LY85b}. Let $\mathbb R^m=\mathbb R^s\oplus\mathbb R^c\oplus\mathbb R^u$
	correspond to $E^s(x)\oplus E^c(x)\oplus E^u(x)$, and equip
	it with the norm
	$	|(v^s,v^c,v^u)|
	=\max\{|v^s|,|v^c|,|v^u|\}.$
	For $x\in\Gamma$, let $\lambda^+(x)$ be the smallest
	positive Lyapunov exponent when $E^u(x)\ne\{0\}$,
	and let $\lambda^-(x)$ be the largest negative
	Lyapunov exponent when $E^s(x)\ne\{0\}$.
	We use the convention $\min\emptyset=+\infty$ and omit
	conditions involving a trivial subspace.
	\begin{definition}[Lyapunov charts]\label{def 3.7}
		Let $f:M\to M$ be a $C^{1+\alpha}$ diffeomorphism and
		$\mu\in \mathcal M_f(M)$. For  $x\in \Gamma$ and  $0<\epsilon<\frac1{100}
		\min\{\lambda^+(x),-\lambda^-(x)\}$, there exists a measureable $\ell: \Gamma\to [1,\infty)$ with $\ell(f^\pm x)\leq e^{\epsilon}\ell(x)$.
		\emph{A Lyapunov chart at $x$} is an embedding
		\begin{align*}
			\psi_x:B(0,\ell(x)^{-1}) \longrightarrow M
		\end{align*}
		with the following properties:
		\begin{enumerate}
			\item $\psi_x(0)=x$, and $D_0\psi_x$ takes
			$\mathbb R^s$, $\mathbb R^c$ and $\mathbb R^u$
			onto $E^s(x)$, $E^c(x)$ and $E^u(x)$.

			\item  Define the connecting map
			$\widetilde{f}_x=\psi_{fx}^{-1}\circ f\circ\psi_x$.  Then it satisfyies the following:
			\begin{align*}
				e^{\lambda^+(x)-\epsilon}|v|	&\leq |D\widetilde{f}_x(0) v|,
				&&v\in\mathbb R^u,\\
				e^{-\epsilon}|v|
				&\leq |D\widetilde{f}_x(0) v|\leq e^\epsilon|v|,
				&&v\in\mathbb R^c,\\
				|D\widetilde{f}_x(0) v|&\leq e^{\lambda^-(x)+\epsilon}|v|,
				&&v\in\mathbb R^s.
			\end{align*}

			\item The Lipschitz constant of
			$\widetilde f_x-D\widetilde f_x(0)$
			on the chart domain is at most $\epsilon$.

			\item For any $v,w\in B(0,\ell(x)^{-1})$, there exists a constant $K$   such that
			\begin{align*}
				K^{-1} |v-w|
				\leq d\bigl(\psi_x(v),\psi_x(w)\bigr)
				\leq \ell(x)|v-w|.
			\end{align*}
		\end{enumerate}
	\end{definition}

	\begin{lemma}\label{lem 3.8}
		let $f$ be a $C^{1+\alpha}$ diffeomorphism of a closed Riemannian
		manifold $M$ with $\dim M=m$.	Let  $\mu\in\mathcal E_f(M)$ and fix $r>0$. The for
		each  $	0<\epsilon<\min\left\{\frac{r}{4},\frac{1}{100}
		\min_{\lambda_i\ne0}|\lambda_i|\right\}$,
		there exist a Borel set $A\subseteq M$ with $\mu(A)>\frac{3}{4}$ and
		constants $C_3,C_4>0$ depending on $\epsilon$ such that the following holds:
		for every sufficiently large $n$ and  $x\in A$, there exists a
		finite set $P_n(x)\subseteq M$ such that for each distinct $y,z \in P_n(x)$,
		\begin{itemize}
			\item $d_n(x,y)
			\leq C_3 e^{-\epsilon n}$;
			\item $\#P_n(x)
			\geq C_4 e^{m(r-4\epsilon)n}$;
			\item $d_n(y,z)\geq4e^{-nr}$.
		\end{itemize}
	\end{lemma}

	\begin{proof}
		Choose $0<\delta<\epsilon$ sufficiently small and use the
		Lyapunov charts  with parameter
		$\delta$.
		Let $K\geq1$ be the constant in  Definition~\ref{def 3.7}\textup{(4)}, and let $D\geq1$
		be a Lipschitz constant for $f$.
		\begin{step1}
			Choose a sufficiently small  ball $B(0,u(z))$  in each chart so that the connecting map $\widetilde{f}_z$ is  well-defined and define $A$.
		\end{step1}

		By Definition~\ref{def 3.7}, since $\psi_{fz}(0)=fz$, one can show that
		\begin{align*}
			B\bigl(fz,(2K\ell(fz))^{-1}\bigr)\subset \psi_{fz}\left(B(0,(2\ell(fz))^{-1})\right)
		\end{align*}
		Choose a constant $c\in (0,\frac{1}{2})$ satisfying  $Dc e^\delta<\frac1{2K}$
		and take $u(z)=\frac{c}{\ell(z)^2}.$ If $|v|<u(z)$,
		then
		\begin{align*}
			d\bigl(f\psi_z(v),fz\bigr)
			\leq D\ell(z)|v|
			<\frac{Dc}{\ell(z)}
			<\frac1{2K\ell(fz)}.
		\end{align*}
		Thus $\psi^{-1}_{fz}(f\psi_z(v))$ is well-defined.
		In particular, the connecting map $\widetilde{f}_z$ is defined on
		$B(0,u(z))$. Also, for any $j\geq 0$, one has $u(f^jz)\geq u(z)e^{-2\delta j}$.
		Define
		\begin{align*}
			A_L=\{x\in\Gamma:\ell(x)\leq L\}.
		\end{align*}
		Then there exists  $L<\infty$ such that $\mu(A_L)>3/4$. Take $A:=A_L$.

		\begin{step2}
			Construct the finite set $P_n(x)$.
		\end{step2}

		Fix  $x\in A$.
		Consider the splitting
		$T_xM=E^s(x)\oplus E^c(x)\oplus E^u(x)$. Denote by
		$a=\dim E^s+\dim E^c$ and $b=\dim E^u=m-a$.
		Take $\chi=\frac{\lambda^+}{2}$ if $b>0$, otherwise conditions
		involving $\chi$ are omitted. Then we group the stable and central
		coordinates into $\mathbb R^a$ and the unstable coordinates into
		$\mathbb R^b$.
		Notice that the connecting map on
		$B(0,u(f^jx))$ can be written as
		\begin{align*}
			\widetilde f_{f^jx}(v)
			=
			\begin{pmatrix}
				A_j&0\\
				0&B_j
			\end{pmatrix}v+g_j(v),
			\qquad
			\|A_j\|\leq e^\delta,\qquad
			\|B_j^{-1}\|\leq e^{-\chi}\quad\text{if }b>0,
		\end{align*}
		where  $\operatorname{Lip}(g_j)\leq\delta$. Since $g_j$
		is Lipschitz, we can extend it to the whole space (still denoted by $g_j$) with Lipschitz constant
		at most $C_m\delta$, where $C_m$ depends only on $m$. Choose
		$\delta$ small enough that $q_*<1$, where
		\begin{align*}
			q_*=
			\begin{cases}
				\max\left\{
				e^{\delta-2\epsilon}+C_m\delta e^{-2\epsilon},
				e^{-\chi}(e^{2\epsilon}+C_m\delta)
				\right\}, & b>0,\\[2mm]
				e^{\delta-2\epsilon}+C_m\delta e^{-2\epsilon},
				& b=0.
			\end{cases}
		\end{align*}
		For any $j\geq0$, one can show that
		\begin{align*}
			u(f^jx)\geq\frac{c}{L^2}e^{-2\delta j},
			\quad
			\ell(f^jx)\leq Le^{\delta j}
		\end{align*}

	For $v=(v_0,\ldots,v_{n-1})$, write
	$v_j=(p_j,q_j)\in\mathbb R^a\oplus\mathbb R^b$ and set
	\begin{align*}
		\|v\|_*=
		\max_{0\leq j<n}e^{-2\epsilon j}
		\max\{|p_j|,|q_j|\}.
	\end{align*}
	For boundary data $p_0=\xi$ and $q_{n-1}=\zeta$,
	define $T$ by
	\begin{align*}
		(Tv)^p_0&=\xi,
		&
		(Tv)^p_j&=A_{j-1}p_{j-1}+g_{j-1}^p(v_{j-1})
		&&(1\leq j<n),\\
		(Tv)^q_{n-1}&=\zeta,
		&
		(Tv)^q_i&=B_i^{-1}\bigl(q_{i+1}-g_i^q(v_i)\bigr)
		&&(0\leq i<n-1),
	\end{align*}
		where $(Tv)^p_j$ and $(Tv)^q_i$ denote  the $p$-component and $q$-component of the $j$-th element of  $Tv$.
		Then one can show that
		\begin{align*}
			\|T(v)-T(w)\|_*\leq q_*\|v-w\|_*
		\end{align*}
		By the contraction mapping principle, $T$ has a unique fixed point $v$.
		Set
		$R_n=\frac{(1-q_*)c}{4L^2}e^{-4\epsilon n}.$
		If $\max\{|\xi|,|\zeta|\}\leq R_n$, then $\|v\|_*\leq\frac{R_n}{1-q_*}$ and for $0\leq j<n$,
		\begin{align*}
			|v_j|
			\leq\frac{c}{4L^2}e^{-4\epsilon n+2\epsilon j}
			<\frac{c}{L^2}e^{-2\delta j}
			\leq u(f^jx).
		\end{align*}
		Since  $v_j\in B(0,u(f^jx))$ ,   the point
		$y=\psi_x\bigl(v_0\bigr)$
		satisfies
		$f^jy=\psi_{f^jx}\bigl(v_j\bigr)$
		for any $0\leq j<n$. It implies
		\begin{align*}
			d(f^jx,f^jy)
			\leq\ell(f^jx)|v_j|
			\leq\frac{c}{4L}e^{-4\epsilon n+(2\epsilon+\delta)j}
			\leq\frac{c}{4L}e^{-\epsilon n}.
		\end{align*}
		Hence $d_n(x,y)\leq C_3e^{-\epsilon n}$,
		where $C_3=\frac{c}{4L}$ and $y$ depends on the choice of $\xi$ and $\zeta$.

		Choose a grid $\mathcal G_n$ of spacing $4Ke^{-nr}$ in
		$B_{\mathbb R^a}(0,R_n)\times B_{\mathbb R^b}(0,R_n)$.
		Let $P_n(x)$ consist of the points $y=\psi_x(v_0)$ obtained
		from the boundary $(\xi,\zeta)\in\mathcal G_n$.
		Then for every $y\in P_n(x)$,
		\begin{align*}
			d_n(x,y)\leq C_3e^{-\epsilon n}.
		\end{align*}
		For two distinct $(\xi,\zeta), (\xi',\zeta')\in \mathcal G_n$, one has $\max\{|\xi-\xi'|,|\zeta-\zeta'|\}
		\leq Kd_n(y,y').$
		Thus for distinct points $y,z$ of $P_n(x)$, $d_n(y,z)\geq4e^{-nr}.$

		Since $R_n=\frac{(1-q_*)c}{4L^2}e^{-4\epsilon n}$,
		the grid contains at least
		\begin{align*}
			c_m\left(\frac{R_n}{4Ke^{-nr}}\right)^m
			\geq C_4e^{m(r-4\epsilon)n}
		\end{align*}
		points for all sufficiently large $n$, where
		$C_4>0$ is independent of $x,n$. This proves the lemma.
	\end{proof}

	They we can give the lower bound of $r$-neutralized PS entropy.
	\begin{proof}[Proof of Theorem~\ref{thm 1.4}]
		We first show that for any $\mu\in \mathcal E_f(M),$
		\begin{align}\label{equ 3.8}
			\underline h_\mu^{PS,r}(f,M)\geq h_\mu(f)+r\dim M.
		\end{align}

		Fix an open neighborhood $F$ of $\mu$ and choose a neighborhood
		$F_0$ of $\mu$ with $\overline{F_0}\subseteq F$. Fix $\tau>0$ and choose
		\[
		0<\epsilon<
		\min\left\{\frac r4,\,
		\frac1{100}\min_{\lambda_i\ne0}|\lambda_i|\right\}.
		\]
 By the Brin--Katok formula \cite{BK83}, there are
		$\delta>0$ and a Borel set $A_0$ with
		$\mu(A_0)>3/4$ such that, for all sufficiently large $n$ and $x\in A_0$,
		\begin{align*}
			\mu(B_n(x,4\delta))
			\leq e^{-n(h_\mu(f)-\tau)}
			.
		\end{align*}
		Let $A$ be the set given by Lemma~\ref{lem 3.8}.
		Since $\mu(M_{n,F_0})\to1$, for all sufficiently large
		$n$, the set
		$G_n=A_0\cap A\cap M_{n,F_0}$
		satisfies $\mu(G_n)>1/4$. Let $E_n$ be a maximal $(n,4\delta)$-separated subset of $G_n$. Then $G_n\subset \bigcup_{x\in E_n}B_n(x,4\delta)$. Hence
		$$\mu(G_n)
		\leq \#E_n\,e^{-n(h_\mu(f)-\tau)}$$		and
		$	\#E_n\geq\frac14e^{n(h_\mu(f)-\tau)}.$

		For each $x\in E_n$, consider the set $P_n(x)$ in
		Lemma~\ref{lem 3.8}. For $y\in P_n(x)$, one has $d_n(x,y)\leq C_3e^{-\epsilon n}$. Since $\overline{F_0}\subseteq F$ and
		$d_n(x,y)\leq C_3e^{-\epsilon n}$ for $y\in P_n(x)$,
		the empirical measures $\mathcal E_n(y)$ belong to $F$
		uniformly for all sufficiently large $n$. Hence
		\begin{align*}
			P_n:=\bigcup_{x\in E_n}P_n(x)\subseteq M_{n,F}.
		\end{align*}
		By Lemma~\ref{lem 3.8}, each $P_n(x)$ is
		$(n,4e^{-nr})$-separated. If $x,x'\in E_n$ are distinct,
		$y\in P_n(x)$ and $y'\in P_n(x')$, then
		$	d_n(y,y')
		>4e^{-nr} $
		for  sufficiently large $n$. Thus  $P_n(x)$
		are pairwise disjoint, and $P_n$ is
		$(n,4e^{-nr})$-separated. Consequently,
		\begin{align*}
			\#P_n
			=\sum_{x\in E_n}\#P_n(x)
			\geq C_4e^{m(r-4\epsilon)n}\#E_n
			\geq \frac{C_4}{4}
			e^{n(h_\mu(f)-\tau+m(r-4\epsilon))}.
		\end{align*}
		Since each Bowen ball of radius $e^{-nr}$ contains at most
		one point of $P_n$, we have
		\begin{align*}
			r_n(f,M_{n,F},e^{-nr})
			\geq \#P_n
			\geq \frac{C_4}{4}
			e^{n(h_\mu(f)-\tau+m(r-4\epsilon))}.
		\end{align*}
		Letting $n\to\infty$ gives
		\begin{align*}
			\liminf_{n\to\infty}
			\frac{1}{n}\log r_n(f,M_{n,F},e^{-nr})
			\geq h_\mu(f)-\tau+m(r-4\epsilon).
		\end{align*}
		which implies \eqref{equ 3.8}.

		Let $F$ be an open neighbourhood of $\mu$. For every $\nu\in F\cap\mathcal E_f(M)$, choose a neighborhood $G$ of $\nu$ with $G\subseteq F$. Then
		\begin{align*}
			\liminf_{n\to\infty}\frac{1}{n}\log r_n(f,M_{n,F},e^{-nr})
			\geq\underline h_\nu^{PS,r}(f,M)
			\geq h_\nu(f)+mr.
		\end{align*}
		Taking the supremum over $\nu\in F\cap\mathcal E_f(M)$ and then the infimum over $F\ni\mu$, we obtain
		\begin{align*}
			\underline h_\mu^{PS,r}(f,M)
			\geq mr+\inf_{F\ni\mu}\sup_{\nu\in F\cap\mathcal E_f(M)}h_\nu(f)
			=mr+h_\mu^*(f),
		\end{align*}
		which finishes the proof.
	\end{proof}

	By Pesin's
	entropy formula \cite{Pes77}, we have the following resulT.
	\begin{corollary}
		Let
		$\mu\in\mathcal E_f(M)$ be an SRB measure.  Then we have
		\begin{align*}
			\underline h_\mu^{PS,r}(f,M)
			=\overline h_\mu^{PS,r}(f,M)
			=h_\mu(f)+mr=h_\mu^*(f)+mr.
		\end{align*}
	\end{corollary}

	\subsection{Eaxct formulas: $C^\infty$ and $C^1$ Anosov diffeomorphism}
	The following  is  a direct result of Proposition~\ref{prop 3.5} and Theorems \ref{thm 1.3}, \ref{thm 1.4}.
	\begin{corollary}\label{cor 3.10}
		Let $f$ be a $C^\infty$ diffeomorphism of a closed Rimannian $m$-manifold. Then 	 for every $r>0$ and $\mu\in\mathcal E_f(M)$,
		\begin{align*}
			\underline h_\mu^{PS,r}(f,M)
			=&\overline h_\mu^{PS,r}(f,M)
			=h_\mu(f)+r\dim M,\\
			\underline h_{top}^r(f,M)
			=&\overline h_{top}^r(f,M)
			=h_{top}(f)+r\dim M.
		\end{align*}
	\end{corollary}

	\begin{proof}
		We only need to prove the upper boud. 	Since $\chi_{\max}^+(f,\mu)$ is finite, letting $s\to\infty$ in Proposition~\ref{prop 3.5} gives
		\begin{align*}
			\overline h_\mu^{PS,r}(f,M)
			\leq mr+h_\mu(f)
		\end{align*}
		By the variational principles, one also have
		\begin{align*}
			\overline h_{top}^r(f,M)
			=\max_{\mu\in\mathcal M_f(M)}
			\overline h_\mu^{PS,r}(f,M)
			\leq h_{top}(f)+mr,
		\end{align*}
		which finishes the proof.
	\end{proof}

	Recall that a  diffeomorphism $f$ is Anosov if there is a continuous $Df$-invariant splitting $TM=E^s\oplus E^u$ and constants $C>0$ and $0<\lambda<1$ such that  for every $n\geq0$,
	\begin{align*}
		\|Df^n v\|&\leq C\lambda^n\|v\|
		&& (v\in E^s),\\
		\|Df^{-n}v\|&\leq C\lambda^n\|v\|
		&& (v\in E^u).
	\end{align*}

	\begin{lemma}\label{lem 3.11}
		Let $f:M\to M$ be a $C^1$ Anosov diffeomorphism of a closed
		Riemannian $m$-manifold. Then there exist $\delta>0$ and $C>0$
		such thatfor every $x\in M$, $n\geq1$ and $0<\epsilon<1$,we have
		\begin{align*}
			r_n\bigl(f,B_n(x,\delta),\epsilon\bigr)
			\leq C\epsilon^{-m}.
		\end{align*}
	\end{lemma}

	\begin{proof}
		Fix $\theta\in(\lambda,1)$. We define the new forms as follows:
		\begin{align*}
			\|v^s\|_{*,x}
			=\sum_{k=0}^{\infty}
			\theta^{-k}\|Df_x^kv^s\|,
			\quad
			\|v^u\|_{*,x}
			=\sum_{k=0}^{\infty}
			\theta^{-k}\|Df_x^{-k}v^u\|.
		\end{align*}
		For $v=v^s+v^u\in E_x^s\oplus E_x^u$, set $\|v\|_{*,x}
		=\max\{\|v^s\|_{*,x},\|v^u\|_{*,x}\}.$
		Then one can show
		\begin{align*}
			\|Df_xv^s\|_{*,fx}
			\leq\theta\|v^s\|_{*,x},
			\quad
			\|Df_x^{-1}v^u\|_{*,f^{-1}x}
			\leq\theta\|v^u\|_{*,x}.
		\end{align*}
		For $x\in M$ and $k\geq0$, define
		\begin{align*}
			F_{x,k}(v)
			=\exp_{f^{k+1}x}^{-1}
			\bigl(f(\exp_{f^kx}v)\bigr),\quad
			g_{x,k}(v)
			=F_{x,k}(v)-Df_{f^kx}v.
		\end{align*}
		Since $Dg_{x,k}(0)=0$ and $f$ is $C^1$, by the compactness of $M$, there exists $\kappa<\frac{1-\theta}{2}$ such that
		for all $x\in M$ and $k\geq 0$,
		$\operatorname{Lip}_*(g_{x,k})\leq\kappa.$
		Choose $\delta>0$ small enough such that for any $y\in B_n(x,\delta)$ and $0\leq k<n$, these points $f^ky$
		lie in exponential chart of $f^kx$. Then  $v_k(y):=\exp_{f^kx}^{-1}(f^ky)\in T_{f^kx}M$.

		Write $v_k(y)=p_k(y)+q_k(y)$, where
		$p_k(y)\in E^s_{f^kx}$ and $q_k(y)\in E^u_{f^kx}$. Take $\xi_y=p_0(y)$, $\zeta_y=q_{n-1}(y)$.
		For any $y,z\in B_n(x,\delta)$, we define
		$	W:=\max_{0\leq k<n}
		\|v_k(y)-v_k(z)\|_{*,f^kx}.$
		Then for $0\leq k<n$, one has
		\begin{align*}
			\|p_k(y)-p_k(z)\|_{*,f^kx}
			&\leq\theta^k\|\xi_y-\xi_z\|_{*,x}
			+\frac{\kappa}{1-\theta}W,\\
			\|q_k(y)-q_k(z)\|_{*,f^kx}
			&\leq\theta^{n-1-k}
			\|\zeta_y-\zeta_z\|_{*,f^{n-1}x}
			+\frac{\theta\kappa}{1-\theta}W.
		\end{align*}
		By the choice of  $\kappa<(1-\theta)/2$, one can show
		\begin{align*}
			W\leq
			2\max\{\|\xi_y-\xi_z\|_{*,x},
			\|\zeta_y-\zeta_z\|_{*,f^{n-1}x}\}.
		\end{align*}
		Since the exponential maps are uniformly Lipschitz on
		the chosen charts, there exists $C_0>0$
		such that
		\begin{align*}
			d_n(y,z)
			\leq C_0
			\max_{0\leq k<n}
			\|v_k(y)-v_k(z)\|_{*,f^kx}
			\leq 2C_0
			\max\{\|\xi_y-\xi_z\|_{*,x},
			\|\zeta_y-\zeta_z\|_{*,f^{n-1}x}\}.
		\end{align*}

		The endpoint coordinates $(\xi_y,\zeta_y)$ range over a uniformly bounded
		subset of $E_x^s\times E_{f^{n-1}x}^u$ of dimension $m$.
		We partition this subset into at most
		$C'\frac{2C_0}{\epsilon}^m$ cells of diameter less than $\frac{\epsilon}{2C_0}$ in the maximum norm, where
		$C'$ only depends on $M$. We choose a point of $B_n(x,\delta)$ from each  cell, by which
		these points form an
		$(n,\epsilon)$-spanning set. Hence
		\begin{align*}
			r_n(f,B_n(x,\delta),\epsilon)
			\leq C'(2C_0)^m\epsilon^{-m}.
		\end{align*}
		Taking $C=C'(2C_0)^m$ finishes the proof.
	\end{proof}

	We now prove the exact formula for some special diffeomorphisms.

	\begin{proof}[Proof of Theorem~\ref{thm 1.5}]
		The case of $C^{\infty}$ directly follows from Corollary \ref{cor 3.10}. Now assume $f$ is $C^1$ Anosov 	diffeomorphism. Fix $r>0$, $\mu\in\mathcal M_f(M)$, and $\eta>0$.
		Let $\delta_*>0$ be the radius in Lemma~\ref{lem 3.11}.
		By Lemma~\ref{lem 3.2}, we may choose
		$0<\epsilon<\delta_*$ and a neighborhood $F\ni\mu$
		such that
		\begin{align*}
			\limsup_{n\to\infty}\frac1n
			\log s_n(f,M_{n,F},\epsilon)
			\leq h_\mu(f)+\eta.
		\end{align*}
		Take an  $(n,\epsilon)$-separated set
		$E_n\subset M_{n,F}$ with the largest cardinality. Then one has $M_{n,F}\subset\bigcup_{x\in E_n}B_n(x,\delta_*)$.
		Apply Lemma~\ref{lem 3.11} to Bowen balls with radius $\frac{e^{-nr}}{2}$. By choosing points in $M_{n,F}$,
		we can show that
		\begin{align*}
			r_n(f,M_{n,F},e^{-nr})
			\leq C2^me^{mnr}s_n(f,M_{n,F},\epsilon).
		\end{align*}
		Consequently,
		\begin{align*}
			\overline h_\mu^{PS,r}(f,M)
			\leq h_\mu(f)+\eta+mr.
		\end{align*}
		Letting $\eta\to 0$ proves the upper bound.

		We next show that for any $\mu\in \mathcal E_f(M)$,
		\begin{align}\label{equ 3.9}
			\underline h_\mu^{PS,r}(f,M)\geq h_\mu(f)+mr.
		\end{align}
		Fix a neighborhood
		$F$ of $\mu$ and choose $F_0\ni\mu$ with $\overline{F_0}\subset F$.
		There exists $\gamma>0$ such that for every $n$, if $x\in M_{n,F_0}$ and $d_n(x,y)<\gamma$, then $y\in M_{n,F}.$
		As in the proof of Theorem~\ref{thm 1.3}, for every
		$\tau>0$ we may choose  $\delta>0$ and an
		$(n,4\delta)$-separated set $E_n\subset M_{n,F_0}$
		such that $	\#E_n\geq \frac14e^{n(h_\mu(f)-\tau)}$
		for all sufficiently large $n$.

	Use the adapted norm and maps $F_{x,k}$ from
	Lemma~\ref{lem 3.11}. Let $\pi_k^s,\pi_k^u$ denote
	the projections onto $E^s_{f^kx},E^u_{f^kx}$.
	For boundary data
	$\xi\in E_x^s$ and $\zeta\in E^u_{f^{n-1}x}$,
	consider
	\[
	\left\{
	\begin{aligned}
		& p_0=\xi,\qquad q_{n-1}=\zeta,\\
		& p_{k+1}=A_kp_k+\pi^s_{k+1}g_{x,k}(p_k+q_k),\\
		& q_k=(A_k|_{E^u_{f^kx}})^{-1}
		\bigl(q_{k+1}-\pi^u_{k+1}g_{x,k}(p_k+q_k)\bigr),
		\qquad 0\leq k<n-1.
	\end{aligned}
	\right.
	\]
	The estimates in Lemma~\ref{lem 3.11} make these
	equations a contraction on a fixed small ball of orbit
	segments. Thus, for some $R_0>0$ independent of $x,n$,
	every boundary datum of norm less than $R<R_0$
	determines a unique point $y_{\xi,\zeta}$, with $d_n(x,y_{\xi,\zeta})\leq CR$ and
	\begin{align*}
		\max\{\|\xi-\xi'\|_{*,x},
		\|\zeta-\zeta'\|_{*,f^{n-1}x}\}
		\leq Ld_n(y_{\xi,\zeta},y_{\xi',\zeta'}),
	\end{align*}
	where $C,L$ are independent of $x,n$.
		Choose $0<R<C^{-1}\min\{\gamma,\delta\}$.
		For each $x\in E_n$, let $\mathcal G_n(x)$ be a
		$4Le^{-nr}$-separated grid in
		\begin{align*}
			\left\{(\xi,\zeta)\in E_x^s\times E_{f^{n-1}x}^u:
			\max\{\|\xi\|_{*,x},
			\|\zeta\|_{*,f^{n-1}x}\}<R\right\},
		\end{align*}
		and set
		\begin{align*}
			P_n(x)=\{y_{\xi,\zeta}:(\xi,\zeta)\in\mathcal G_n(x)\}.
		\end{align*}
		As in Lemma \ref{lem 3.8}, one has  $P_n(x)\subset B_n(x,CR)\cap M_{n,F}$ and $\#P_n(x)\geq c_Re^{mrn}$, where $c_R$ is independent of $x$ and $n$. Moreover,
		$P_n(x)$ is $(n,4e^{-nr})$-separated.
		If distinct points $x, x'\in E_n$, then for $y\in P_n(x)$ and $ y'\in P_n(x')$,
		\begin{align*}
			d_n(y,y')\geq 4\delta-2CR>2\delta>2e^{-nr}.
		\end{align*}
		Thus $\bigcup_{x\in E_n}P_n(x)$ is
		$(n,2e^{-nr})$-separated for all sufficiently large $n$, and
		\begin{align*}
			r_n(f,M_{n,F},e^{-nr})
			\geq \sum_{x\in E_n}\#P_n(x)
			\geq \frac{c_R}{4}
			e^{n(h_\mu(f)-\tau+mr)}.
		\end{align*}
		Since $F$ and $\tau$ were arbitrary, one can show
		\begin{align*}
			\underline h_\mu^{PS,r}(f,M)\geq h_\mu(f)+mr,
		\end{align*}
		which proves \eqref{equ 3.9}

		The topological formula follows from the measure formula and the variational principle.
	\end{proof}

	Notice that for every ergodic measure of an Anosov diffeomorphism, one can show
	\begin{align*}
		\underline h_\mu^{PS,r}(f,M)=	\overline h_\mu^{PS,r}(f,M)= h_\mu(f)+mr.
	\end{align*}
	Combining with the work of \cite{DQ25}, it is nartural to consider whether the equality  holds for any $C^{1+\alpha}$ diffeomorphism with ergodic hyperbolic measures.
	The following example gives a negative answer, even if we strengthen the system to
	$C^2$.

	\begin{example}\label{exam 3.12}
		Let $M$ be a smooth compact $C^r$  ($r\geq 2$) surface, and  $f:M\to M$ be a
		$C^2$ diffeomorphism
		constructed in \cite[Section~6]{DN05}. Choose a hyperbolic saddle periodic point $p$
		and take
		\begin{align*}
			\mu_p=\frac{1}{\tau(p)}
			\sum_{i=0}^{\tau(p)-1}\delta_{f^ip}
		\end{align*}
		be the uniform measure on its orbit, where $\tau(p)$ is
		the least period of $p$.
		Then for every $r>0$,
		\begin{align*}
			h_{\mu_p}(f)+2r
			=2r
			<\underline h_{\mu_p}^{PS,r}(f,M)
			\leq\overline h_{\mu_p}^{PS,r}(f,M),
		\end{align*}
		where $\mu_p$ is ergodic and hyperbolic.
	\end{example}

	\begin{proof}
		It is easy to show $h_{\mu_p}(f)=0$. Then we are left to show $\underline h_{\mu_p}^{PS,r}(f,M)>2r$ for any $r>0$.

		By equations (56), (57) of \cite{DN05}, there exist ergodic measures
		$\nu_j$ supported on hyperbolic basic sets such that $\nu_j\longrightarrow\mu_p$ and
	$	h_{\nu_j}(f)\longrightarrow\frac{\chi(p)}2>0,$
		where  $\chi(p)$ is the
		positive Lyapunov exponent of $p$.
		Since $\nu_j\to\mu_p$ and each $\nu_j$ is ergodic,
		the definition of $h_{\mu_p}^*(f)$ gives
		\begin{align*}
			h_{\mu_p}^*(f)
			\geq \limsup_{j\to\infty}h_{\nu_j}(f)
			=\frac{\chi(p)}2.
		\end{align*}
		Thus Theorem~\ref{thm 1.4} yields
		\begin{align*}
			\underline h_{\mu_p}^{PS,r}(f,M)
			\geq 2r+h_{\mu_p}^*(f)
			\geq 2r+\frac{\chi(p)}2>2r.
		\end{align*}
	\end{proof}




\section{Estimates and exact formulas on compact invariant sets}
\label{sec 4}
In this section we establish exact formulas on mixing surface hyperbolic
basic sets and compare $r$-neutralized PS entropy with local entropy.

\subsection{Exact formulas on mixing surface hyperbolic basic sets}
Let $f:M\to M$ be a $C^{1+\alpha}$ diffeomorphism of a closed
Riemannian $m$-manifold and let $\Lambda\subset M$ be compact and invariant.
Then for each ergodic hyperbolic measure $\mu$ on $\Lambda$, by Theorem \ref{thm 1.3} and \cite{DQ25}, one has
\begin{equation}\label{equ 4.1}
	\begin{split}
		h_\mu(f,\Lambda)+r\dim_H\mu
		\leq\underline h_\mu^{PS,r}(f,\Lambda)
		\leq\overline h_\mu^{PS,r}(f,\Lambda)
		\leq r \dim M+\Phi^*_1(f,\mu),
	\end{split}
\end{equation}
where  the upper bound follows from
the fact that Proposition \ref{prop 3.5} also holds if restricting to
$\Lambda$.

The bounds in \eqref{equ 4.1} do not determine the contribution
of $\Lambda$ itself.
It is natural to ask whether $\dim_H\Lambda$ can be replaced by $\dim M$.
Different from the classical entropy quantity, where the variational principle
remains valid on compact invariant subsets, the $r$-neutralized PS entropy
may exhibit different behavior when
restricted to compact invariant subset. Moreover, the exact  formula in
Theorem \ref{thm 1.5} may fail,  by Example \ref{exam 2.16}.
Since exact formulas require additional uniform structure, we restrict to
mixing surface hyperbolic basic sets, whose local product structure yields
uniform estimates inside Bowen balls.

Throughout this subsection, we consider the $C^{1+\alpha}$ diffeomorphism $f$ of a closed Riemannian surface  $M$.
We first recall some basic definitions as follows:
\begin{definition}
	A nonempty compact $f$-invariant set $\Lambda$ is
	\emph{hyperbolic} if there is a continuous $Df$-invariant
	splitting $T_\Lambda M=E^s\oplus E^u$ and constants
	$C>0$, $0<\theta<1$ such that, for every $x\in\Lambda$
	and $n\geq0$,
	\begin{align*}
		\|D_xf^n v\|&\leq C\theta^n\|v\|,
		&&v\in E_x^s,\\
		\|D_xf^{-n}v\|&\leq C\theta^n\|v\|,
		&&v\in E_x^u.
	\end{align*}
	It is \emph{locally maximal} if
	$\Lambda=\bigcap_{n\in\mathbb Z}f^n(U)$ for some open
	neighborhood $U$ of $\Lambda$.
	A locally maximal hyperbolic set is a \emph{basic set}
	if $f|_\Lambda$ is topologically transitive.
\end{definition}

Let $\Lambda$ be a mixing hyperbolic
basic set with $\dim E^s=\dim E^u=1$, the stable and
unstable  manifolds at $x$ are defined as
\begin{align*}
	W^s(x)=&\left\{y\in M: \limsup_{n\to \infty}\frac{1}{n}\log d(f^ny,f^nx)<0\right\} ,\\
	W^u(x)=&\left\{y\in M: \limsup_{n\to \infty}\frac{1}{n}\log d(f^{-n}y,f^{-n}x)<0\right\}.
\end{align*}
These are immersed
one-dimensional manifolds tangent at $x$ to $E_x^s$
and $E_x^u$, respectively.
Denote by $\rho^s$ and $\rho^u$ the intrinsic distances
induced by the Riemannian metric. Define the corresponding stable and unstable balls as
\begin{align*}
	B^s(z,\epsilon):=\{y\in W^s(x): \rho^s(z,y)<\epsilon\},\quad B^u(z,\epsilon):=\{y\in W^u(x): \rho^u(z,y)<\epsilon\}
\end{align*}
Choose a sufficiently small uniform radius $t_0>0$ and denote the local stable (unstable) manifold by
\begin{align*}
	W^i_{\mathrm{loc}}(x)
	=\{y\in W^i(x): \rho^i(x,y)<t_0\},
	\quad i=s,u.
\end{align*}
By \cite[Theorem~1.1]{DGGJ19}, for each $x\in \Lambda$, one has
\begin{align}\label{equ 4.2}
	\dim_H\Lambda
	=\dim_B\bigl(\Lambda\cap W^s_{\mathrm{loc}}(x),\rho
	^s\bigr)
	+\dim_B\bigl(\Lambda\cap W^u_{\mathrm{loc}}(x),\rho
	^u\bigr):
	=d_s+d_u,
\end{align}
where $d_s$ and $d_u$ are  independent of $x\in\Lambda$.
For sufficiently small $t>0$, we define
$K^i(x,t)
=\{y\in\Lambda\cap W^i_{\mathrm{loc}}(x):\rho^i(x,y)<t\} $
and the map:
\begin{align*}
	\Phi_x:K^s(x,t)\times K^u(x,t)\longrightarrow\Lambda,
	\quad
	\Phi_x(a,b)=W^u_{\mathrm{loc}}(a)\cap W^s_{\mathrm{loc}}(b).
\end{align*}
By \cite[Remark~1.2]{DGGJ19}, $\Phi_x$ is bi-Lipschitz
with a constant $L\geq1$ independent of $x\in\Lambda$.

Denote by $N(K^i(x,t),\rho^i,\epsilon)$ the minimum
number of balls $B^i(z,\epsilon)$  covering
$K^i(x,t)$, and by $S(K^i(x,t),\rho^i,\epsilon)$
the maximum cardinality of an $\epsilon$-separated
subset of $K^i(x,t)$ with respect to $\rho^i$.
We next give some estiamtes for the covering cardinality of Bowen balls restricted on $\Lambda$.

\begin{lemma}\label{lem 4.2}
	For sufficiently small fixed $\delta>0$ and
	$\eta>0$, there exist two constant $c,C>0$ and $\epsilon_0>0$ such
	that for each $x\in\Lambda$, $n\in \mathbb N$, and
	$\epsilon\in(0,\epsilon_0)$, we have the following estimates:
	\begin{equation*}
		r_n(f,B_n(x,\delta)\cap\Lambda,\epsilon)
		\leq C\epsilon^{-(\dim_H\Lambda+\eta)},\quad
		s_n(f,B_n(x,\delta)\cap\Lambda,\epsilon)
		\geq c\epsilon^{-(\dim_H\Lambda-\eta)}.
	\end{equation*}
\end{lemma}

\begin{proof}
	Fix a sufficiently small $t>0$ and $i=s,u$.
	A finite local product cover and the uniformly bi-Lipschitz
	holonomies compare $K^i(x,t)$, up to fixed changes of scale,
	with finitely many reference $i$-slices. In particular, each
	$K^i(x,t)$ contains the image of a smaller reference slice,
	with constants independent of $x$.
	By \eqref{equ 4.2}, for every $\gamma>0$ there exist
	$c_{t,\gamma},C_{t,\gamma},\epsilon_{t,\gamma}>0$ such that for every
	$x\in\Lambda$ and $0<\epsilon<\epsilon_{t,\gamma}$,
	\begin{equation}\label{equ 4.3}
		N(K^i(x,t),\rho^i,\epsilon)\leq C_{t,\gamma}\epsilon^{-(d_i+\gamma)},
		\qquad
		S(K^i(x,t),\rho^i,\epsilon)\geq c_{t,\gamma}\epsilon^{-(d_i-\gamma)}.
	\end{equation}

	For $(a,b)\in K^s(x,t)\times K^u(f^{n-1}x,t)$, define
	\begin{equation*}
		\Theta_{n,x}(a,b)=W^s_{\mathrm{loc}}\bigl(f^{-(n-1)}b\bigr)\cap W^u_{\mathrm{loc}}(a).
	\end{equation*}
	For sufficiently small $t$, uniform hyperbolicity and
	local product structure ensure that $\Theta_{n,x}$
	is well-defined with values in $\Lambda$ for all $n,x$.
	Uniform contraction along stable and unstable plaques,
	together with the bi-Lipschitz coordinate projections
	at times $0$ and $n-1$, gives a constant $C_0\geq1$,
	independent of $x,n$, such that, for
	$z=\Theta_{n,x}(a,b)$ and
	$z'=\Theta_{n,x}(a',b')$,
	\begin{equation}\label{equ 4.4}
		C_0^{-1}\max\{\rho^s(a,a'),\rho^u(b,b')\}
		\leq d_n(z,z')
		\leq C_0\max\{\rho^s(a,a'),\rho^u(b,b')\}.
	\end{equation}

	Consider the rectangle
	\begin{equation*}
		R_{n,x}(t)=\Theta_{n,x}\bigl(K^s(x,t)\times K^u(f^{n-1}x,t)\bigr).
	\end{equation*}
	Fix sufficiently small $\delta>0$.  There exist constants $a_0,A_0>0$
	(independent of $n$) such that
	\begin{equation}
		R_{n,x}(a_0\delta)\subset B_n(x,\delta)\cap\Lambda\subset R_{n,x}(A_0\delta).
	\end{equation}
	Indeed, we can  take $0<a_0<\frac{1}{2C_0}$. For the second relation, for any  $y\in B_n(x,\delta)\cap\Lambda$ and $0\leq k<n$, denote by  $q_k=W^s_{\mathrm{loc}}(f^ky)\cap W^u_{\mathrm{loc}}(f^kx)$. Then for  $\delta$ sufficiently small, one has
	$fq_k=q_{k+1}$. Take $a=W^u_{\mathrm{loc}}(y)\cap W^s_{\mathrm{loc}}(x)$ and $b=q_{n-1}=f^{n-1}q_0.$
	Then by the bi-Lipschitz and local product properties, one can find some $A_0>0$ such that
	\begin{equation*}
		\rho^s(x,a)<A_0\delta,\quad
		\rho^u(f^{n-1}x,b)<A_0\delta.
	\end{equation*}
	Then $y=\Theta_{n,x}(a,b)\in R_{n,x}(A_0\delta)$.

	Take $\gamma=\frac{\eta}{2}$. Cover
	$K^s(x,A_0\delta)\times K^u(f^{n-1}x,A_0\delta)$
	by products of balls of radius $\epsilon/(4C_0)$.
	For each product whose image under $\Theta_{n,x}$
	meets $B_n(x,\delta)\cap\Lambda$, choose a point of
	that intersection. By \eqref{equ 4.4}, these points form
	an $(n,\epsilon)$-spanning set. Hence \eqref{equ 4.3} gives
	\begin{align*}
		r_n(f,B_n(x,\delta)\cap\Lambda,\epsilon)
		&\leq
		N\left(K^s(x,A_0\delta),\rho^s,\frac{\epsilon}{4C_0}\right)
		N\left(K^u(f^{n-1}x,A_0\delta),\rho^u,
		\frac{\epsilon}{4C_0}\right)\\
		&\leq C\epsilon^{-(d_s+d_u+\eta)}.
	\end{align*}
	Similarly, one has
	\begin{align*}
		s_n(f,B_n(x,\delta)\cap\Lambda,\epsilon)
		&\geq
		S(K^s(x,a_0\delta),\rho^s,C_0\epsilon)\cdot
		S(K^u(f^{n-1}x,a_0\delta),\rho^u,C_0\epsilon)\\
		&\geq c\epsilon^{-(d_s+d_u-\eta)}.
	\end{align*}
	Then the desired result follows from \eqref{equ 4.2}.
\end{proof}

\begin{proposition}\label{prop 4.3}
	Let $\Lambda$ be a  mixing  hyperbolic basic set  of the  closed Riemannian surface  $M$. For each $\mu\in\mathcal E_f(\Lambda)$ and $r>0$, we have
	\begin{align*}
		\underline h_\mu^{PS,r}(f,\Lambda)
		=\overline h_\mu^{PS,r}(f,\Lambda)
		=&h_\mu(f,\Lambda)+r\dim_H\Lambda,\\
		\underline h_{top}^r(f,\Lambda)
		=\overline h_{top}^r(f,\Lambda)
		=&h_{top}(f,\Lambda)+r\dim_H\Lambda.
	\end{align*}
\end{proposition}

\begin{proof}
	Fix $\mu\in\mathcal E_f(\Lambda)$, $r>0$, and $\eta>0$.
	By Lemma~\ref{lem 3.2}, for sufficiently small $\delta>0$
	there exists a neighborhood $F\ni\mu$ such that
	$\limsup_{n\to\infty}\frac{1}{n}
	\log s_n(f,\Lambda_{n,F},\delta)
	\leq h_\mu(f)+\frac{\eta}{2}.$
	It means that for all sufficiently large $n$, there exists an $(n,\delta)$-separated set $F_n$ of $\Lambda_{n,F}$ with the largest cardinality satiasfying $\#F_n\leq e^{n(h_\mu(f)+\eta)}$. Since $\Lambda_{n,F}\subset\bigcup_{x\in F_n}B_n(x,\delta)$,
	apply Lemma~\ref{lem 4.2} with  $\epsilon=\frac{e^{-nr}}{2}$.
	Then one has
	\begin{align*}
		r_n(f,\Lambda_{n,F},e^{-nr})
		\leq \#F_n\cdot C\cdot 2^{\dim_H\Lambda+\eta}e^{nr(\dim_H\Lambda+\eta)}
	\end{align*}
	Consequently, it shows
	\begin{align*}
		\overline h_\mu^{PS,r}(f,\Lambda)
		\leq h_\mu(f)+\eta+r(\dim_H\Lambda+\eta).
	\end{align*}

	For the lower bound, we fix a neighborhood $F\ni\mu$. Then we can find $\delta>0$ and a
	smaller neighborhood $F_0\subset F$ of $\mu$ such that for any $x\in \Lambda_{n,F_0}$, one has
	\begin{align*}
		B_n(x,\delta)\cap\Lambda\subset\Lambda_{n,F}.
	\end{align*}
	Fix  $\delta$ sufficiently small.
	By Lemma~\ref{lem 3.2}, for all sufficiently large $n$, there is an $(n,4\delta)$-separated set
	$E_n\subset\Lambda_{n,F_0}$ with $\#E_n\geq e^{n(h_\mu(f)-\eta)}.$
	For each $x\in E_n$, we apply
	Lemma~\ref{lem 4.2} with $2e^{-nr}$. Then there exists   an
	$(n,2e^{-nr})$-separated set
	$A_{n,x}\subset B_n(x,\delta)\cap\Lambda$ with
	\begin{align*}
		\#A_{n,x}\geq c\cdot2^{-(\dim_H\Lambda-\eta)}\cdot e^{nr(\dim_H\Lambda-\eta)}.
	\end{align*}
	Each $A_{n,x}$ is disjoint. If $x\ne x'$ in $E_n$, $y\in A_{n,x}$, and
	$y'\in A_{n,x'}$, then
$d_n(y,y')
>2\delta>e^{-nr}$
	for all sufficiently large $n$. Thus
	$\bigcup_{x\in E_n}A_{n,x}$ is
	$(n,e^{-nr})$-separated, and
	\begin{align*}
		s_n(f,\Lambda_{n,F},e^{-nr})\geq \sum_{x\in E_n}\#A_{n,x}\geq c\cdot e^{n(h_\mu(f)-\eta)}\cdot2^{-(\dim_H\Lambda-\eta)}\cdot e^{nr(\dim_H\Lambda-\eta)},
	\end{align*}
	which shows
	\begin{align*}
		\underline h_\mu^{PS,r}(f,\Lambda)
		\geq h_\mu(f)-\eta+r(\dim_H\Lambda-\eta).
	\end{align*}
	Taking the infimum over $F$ and letting $\eta\to 0$, we prove the first formula. The second is a direct result of the first formula.

\end{proof}

\subsection{Further comparison with $r$-neutralized local entropy}

For ergodic hyperbolic measures on $\Lambda$, one has
$\overline{h}_{\mu}^{BK,r}(f,\Lambda)\leq \underline{h}_{\mu}^{PS,r}(f,\Lambda)$.
On the mixing basic sets considered in Proposition~\ref{prop 4.3}, we can give
the exact formula for the $r$-neutralized PS entropy and local entropies by \cite{DQ25}:
\begin{equation*}
	\begin{split}
		\underline h_\mu^{PS,r}(f,\Lambda)-\overline h_\mu^{BK,r}(f,\Lambda)
		=\overline h_\mu^{PS,r}(f,\Lambda)-\overline h_\mu^{BK,r}(f,\Lambda)
		=r(\dim_H\Lambda-\dim_H\mu)\geq 0.
	\end{split}
\end{equation*}
Thus equality holds if and only if
$\dim_H\mu=\dim_H\Lambda$.

The following example shows that equality and strict inequality both occur in the preceding comparison, which are direct consequences of the formulas in Proposition \ref{prop 4.3} and \cite{DQ25}.

\begin{example}\label{ex 4.4}
	Let $f$ be a $C^\infty$ diffeomorphism of a closed Riemannian surface
	with an affine $b$-branch horseshoe $\Lambda$ ($b\geq 2$),
	whose expansion and contraction factors are $e^{\chi^u}$ and
	$e^{-\chi^s}$ with $\chi^u,\chi^s>\log b$, respectively.
	Then $\Lambda$ is a mixing hyperbolic basic set of saddle type, and
	$f|_\Lambda$ is conjugate to the full two-sided shift on $b$ symbols.

	For a probability vector $\boldsymbol w=(w_1,\ldots,w_b)$,
	$w_i>0$, let $\mu_{\boldsymbol w}$ be the corresponding Bernoulli
	measure on $\Lambda$. Then $\mu$ is ergodic and hyperbolic. Set $H(\boldsymbol w)=-\sum_{i=1}^b w_i\log w_i$.
	By \cite{You82} and \cite{MM83}, one has $h_{\mu_{\boldsymbol w}}(f)=H(\boldsymbol w)$ and
	\begin{align*}
		\dim_H\mu_{\boldsymbol w}=\left(\frac1{\chi^u}+\frac1{\chi^s}\right) H(\boldsymbol w),\quad \dim_H\Lambda=\left(\frac1{\chi^u}+\frac1{\chi^s}\right)\log b.
	\end{align*}
	By Proposition~\ref{prop 4.3} and \cite{DQ25},
	for every $r>0$,
	\begin{align*}
		\underline h_{\mu_{\boldsymbol w}}^{PS,r}(f,\Lambda)
		=\overline h_{\mu_{\boldsymbol w}}^{PS,r}(f,\Lambda)
		=&H(\boldsymbol w)+r\left(\frac1{\chi^u}+\frac1{\chi^s}\right)\log b,\\
		\overline h_{\mu_{\boldsymbol w}}^{BK,r}(f,\Lambda)=\underline h_{\mu_{\boldsymbol w}}^{BK,r}(f,\Lambda)
		=&H(\boldsymbol w)+r\left(\frac1{\chi^u}+\frac1{\chi^s}\right)H(\boldsymbol w).
	\end{align*}
	Consequently,
	\[
	\underline h_{\mu_{\boldsymbol w}}^{PS,r}(f,\Lambda)
	-\overline h_{\mu_{\boldsymbol w}}^{BK,r}(f)
	=r\left(\frac1{\chi^u}+\frac1{\chi^s}\right)\bigl(\log b-H(\boldsymbol w)\bigr).
	\]
	In particular, equality holds if and only if $\boldsymbol w$ is uniform
	($w_i=1/b$); otherwise the inequality is strict.
\end{example}

 Without this assumption of ergodicity, by the following example we show that the $r$-neutralized
PS entropy can be strictly smaller than $r$-neutralized local entropy.

\begin{example}\label{exam 4.5}
	Let $f$ be a $C^\infty$ diffeomorphism of a closed Riemannian
	surface $M$ with two identical affine $b$-branch horseshoes
	$\Lambda_1, \Lambda_2$ ($b \ge 2$) in disjoint disks,
	having constant expansion/contraction factors $3b$ and $(3b)^{-1}$.
	Let $\Lambda = \Lambda_1 \sqcup \Lambda_2$ and let $\mu_i$
	be the uniform Bernoulli measure on $\Lambda_i$.
	Then $\mu = \tfrac12(\mu_1+\mu_2)$ is invariant but not ergodic,
	and for every $r>0$,
	\begin{align*}
		\underline h_\mu^{PS,r}(f,\Lambda)
		=\overline h_\mu^{PS,r}(f,\Lambda)
		=0<\log b+\frac{2r\log b}{\log(3b)}=\underline h_\mu^{BK,r}(f)
		=\overline h_\mu^{BK,r}(f).
	\end{align*}
\end{example}

\begin{proof}
	Consider the continuous function	 $\varphi=\mathbf{1}_{\Lambda_2}$ on $\Lambda$. Then
	\begin{align*}
		F=\left\{\nu\in\mathcal M(\Lambda):
		\left|\int\varphi\,\mathrm d\nu-\tfrac12\right|
		<\tfrac14\right\}
	\end{align*}
	is a open neighborhood of $\mu$.
	Since each $\Lambda_i$ is invariant, one can obtain
	$\int\varphi\,\mathrm d\mathcal E_n(x)\in\{0,1\}$ for all $x\in\Lambda, n\geq1$.
	Hence $\Lambda_{n,F}=\emptyset$ for every $n$, which yields
	\begin{align*}
		\underline h_\mu^{PS,r}(f,\Lambda)
		=\overline h_\mu^{PS,r}(f,\Lambda)=0.
	\end{align*}

Since $\operatorname{dist}(\Lambda_1,\Lambda_2)>0$,
for $x\in\Lambda_i$ and all sufficiently large $n$,
$B_n(x,e^{-nr})\cap\Lambda_{3-i}=\varnothing$.
As $\mu_i$ is supported on $\Lambda_i$, it follows that
$\mu(B_n(x,e^{-nr}))
=\frac12\mu_i(B_n(x,e^{-nr})),$
and hence
\begin{align*}
	-\frac1n\log\mu(B_n(x,e^{-nr}))
	=-\frac1n\log\mu_i(B_n(x,e^{-nr}))
	+\frac{\log2}{n}.
\end{align*}
	Since each $\mu_i$ is ergodic and hyperbolic, by \cite{You82}, one has $h_{\mu_i}(f)=\log b$ and $\dim_H\mu_i=\frac{2\log b}{\log(3b)}$.
	Combining with \cite{DQ25}, one can show
	\begin{align*}
		\underline h_\mu^{BK,r}(f)
		=\overline h_\mu^{BK,r}(f)
		=\log b+\frac{2r\log b}{\log(3b)}>0,
	\end{align*}
	which finishes the proof.
\end{proof}

\section*{Acknowledgments}
This work was supported by the National Natural Science Foundation of China (No.~12471184). The
third author was also supported by the Qinglan Project of Jiangsu Province, China.

During the preparation of this manuscript, the authors used GPT-6
to search for relevant references, examine mathematical arguments,
assist in the construction of Example~\ref{exam 3.12}, which shows
that Theorem \ref{thm 1.5}
does not hold for all $C^{1+\alpha}$ diffeomorphisms.
The authors checked and substantially revised all AI-assisted
mathematical arguments and take full responsibility for the final
manuscript.

\section*{Statements and Declarations}
\subsection*{Data Availability}	No datasets were generated or analysed during the current study.
\subsection*{Competing interests}The authors declare no competing interests.

\end{document}